\documentclass[11pt]{article}

\usepackage{fullpage}
\usepackage[T1]{fontenc}

\usepackage{hyperref}
\hypersetup{
    colorlinks=true,
    linkcolor=blue,
    citecolor=red,
    urlcolor=blue,
    pdfborder={0 0 0}
}

\usepackage{xcolor}

\usepackage{graphicx}
\usepackage{hyperref}
\usepackage{amsmath,amsthm,amssymb}
\usepackage{bbm}
\usepackage{tabularx}

\newtheorem{thm}{Theorem}[section]
\newtheorem{conj}[thm]{Conjecture}
\newtheorem{definition}{Definition}[section]
\newtheorem{prop}[thm]{Proposition}

\newtheorem{lemma}[thm]{Lemma}
\theoremstyle{plain}
\newtheorem{rem}[thm]{Remark}

\newcommand{\R}{\mathbb{R}}
\newcommand{\N}{\mathbb{N}}
\renewcommand{\P}{\mathbb{P}}
\newcommand{\E}{\mathbb{E}}
\newcommand{\1}{\mathbbm{1}}

\newenvironment{enum}{\begin{list}{(\roman{enumi})}{\usecounter{enumi}}}{\end{list}}

\newcommand{\Sn}{\mathcal S_n}
\newcommand{\An}{\mathcal A_n}
\newcommand{\X}{\mathfrak{X}}

\newcommand{\nb}[1]{{\color{red} NB comment: #1} }

\usepackage{comment}

\DeclareMathOperator{\var}{Var}

\DeclareMathOperator{\id}{id}

\DeclareMathOperator{\tmix}{\textit{t}_{mix}}

\def\cS{\mathcal{S}}

\def\cM{\mathcal{M}}

\def\cK{\mathcal{K}}

\def\cH{\mathcal{H}}
\def\cG{\mathcal{G}}
\def\cF{\mathcal{F}}

\def\cC{\mathcal{C}}
\def\cB{\mathcal{B}}

\def\tsum{\textstyle \sum}
\def\tprod{\textstyle \prod}

\def\eps{\varepsilon}

\begin{document}

\title{Cutoff for conjugacy-invariant  \\
random walks on the permutation group
}

\author{Nathana\"el Berestycki\footnote{University of Cambridge, Statistical Laboratory. Research supported in part by EPSRC grants EP/GO55068/1,  EP/I03372X/1 and EP/L018896/1.} \and Bat\i{} \c{S}eng\"ul\footnote{Research was carried out in part at the University of Cambridge by EPSRC grant EP/H023348/1 and at the University of Bath by EPSRC grant EP/L002442/1.}}
\date{\today}

\maketitle

\begin{abstract}
We prove a conjecture raised by the work of Diaconis and Shahshahani (1981) about the mixing time of random walks on the permutation group induced by a given conjugacy class. To do this we exploit a connection with coalescence and fragmentation processes and control the Kantorovitch distance by using a variant of a coupling due to Oded Schramm as well as contractivity of the distance.
Recasting our proof in the language of Ricci curvature, our proof establishes the occurrence of a phase transition, which takes the following form in the case of random transpositions: at time $cn/2$,  the curvature is asymptotically zero for $c\le 1$ and is strictly positive for $c>1$.
\end{abstract}

\tableofcontents

\section{Introduction}

\subsection{Main results}
Let $\Sn$ denote the multiplicative group of permutations of $\{1,\dots,n\}$. Let $\Gamma \subset \Sn$ be a fixed conjugacy class in $\Sn$, i.e., $\Gamma = \{g \gamma g^{-1} : g \in \Sn\}$ for some fixed permutation $\gamma \in \Sn$. Alternatively, $\Gamma$ is the set of permutation in $\Sn$ having the same cycle structure as $\gamma$.
Let $X^\sigma=(X_0, X_1, \ldots )$  be discrete-time random walk on $\Sn$ induced by $\Gamma$, started in the permutation $\sigma \in \Sn$, and let $Y^\sigma$ be the associated continuous time random walk. These are the processes defined by
\begin{equation}
  \label{eq:rw_conj_defn}
  \begin{array}{lll}
    X^\sigma_t  &=\sigma \circ \gamma_1 \circ \dots \circ \gamma_t ;& \ \ \  t = 0, 1, \ldots\\
    Y_t^\sigma  &= X^\sigma_{N_t} ; & \ \ \ t \in [0, \infty)
    \end{array}
  \end{equation}
  where $\gamma_1,\gamma_2,\dots$ are i.i.d. random variables which are distributed uniformly in $\Gamma$; and $(N_t, t \ge 0)$ is an independent Poisson process with rate 1.
  Then $Y$ is a Markov chain which converges to an invariant measure $\mu$ as $t \to \infty$. If $\Gamma \subset \An$ (where $\An$ denotes the alternating group) then $\mu$ is uniformly distributed on $\An$ and otherwise $\mu$ is uniformly distributed on $\Sn$. The simplest and most well known example of a conjugacy class is the set $T$ of all transpositions, or more generally of all cyclic permutations of length $k\ge 2$. This set will play an important role in the rest of the paper. Note that $\Gamma$ depends on $n$ but we do not indicate this dependence in our notation.

  \medskip The main goal of this paper is to study the cut-off phenomenon for the random walk $X$. More precisely, recall that the total variation distance $\|X-Y\|_{TV}$ between two random variables $X$, $Y$ taking values in a set $S$ is given by
  \begin{equation}
    \|X-Y\|_{TV} = \sup_{A\subset S} |\P(X\in A)-\P(Y\in A)|.
  \end{equation}
  For $0< \delta <1$,  the mixing time $\tmix(\delta)$ is by definition given by
  \[
    \tmix(\delta) = \inf\{t \geq 0: d_{TV}(t) \le \delta \}
  \]
  where
  \begin{equation}
    d_{TV}(t) = \sup_{\sigma} \| Y^\sigma_t-\mu \|_{TV}
  \end{equation}
  and $\mu$ is the invariant measure defined above.

  In the case where $\Gamma= T$ is the set of transpositions, a famous result of Diaconis and Shahshahani~\cite{diaconis_mixing} is that the cut-off phenomenon takes place at time $(1/2) n \log n$ asymptotically as $n\to \infty$. That is, $\tmix(\delta)$ is asymptotic to $(1/2) n \log n$ for any fixed value of $0< \delta <1$. It has long been conjectured that for a general conjugacy class such that $|\Gamma| = o(n)$ (where here and in the rest of the paper, $|\Gamma|$ denotes the number of non fixed points of any permutation $\gamma \in \Gamma$), a similar result should hold at a time $(1/|\Gamma|) n \log n$. This has been verified for $k$-cycles with a fixed $k\ge 2$ by Berestycki, Schramm and Zeitouni~\cite{berestycki_k}. This is a problem with a substantial history which will be detailed below. The primary purpose of this paper is to verify this conjecture. Hence our main result is as follows.

  \begin{thm}\label{thm:mix}
    Let $\Gamma\subset \Sn$ be a conjugacy class and suppose that $|\Gamma|=o(n)$. Define
    \begin{equation}\label{tmix}
      \tmix:=\frac{1}{|\Gamma|}n \log n.
    \end{equation}
    Then for any $\epsilon>0$,
    \begin{equation}\label{cutoff}
      \lim_{n \rightarrow \infty}d_{TV}((1-\epsilon)\tmix) = 1 \quad \text{ and } \quad \lim_{n \rightarrow \infty}d_{TV}((1+\epsilon)\tmix) =0.
    \end{equation}
  \end{thm}

  \medskip The first limit of \eqref{cutoff} is proved in Appendix \ref{S:lb}. The rest of the paper focuses on the second limit. Our main tool for this result is the notion of discrete Ricci curvature as introduced by Ollivier~\cite{ollivier_ricci}, for which we obtain results of independent interest. We briefly discuss this notion here; however we point out that this turns out to be equivalent to the more well-known \emph{path coupling method} and transportation metric introduced by Bubley and Dyer~\cite{bubley_path} and Jerrum~\cite{jerrum_simple} (see for instance Chapter 14 of the book \cite{peres_mixing_book} for an overview). However we will cast our results in the language of Ricci curvature because we find it more intuitive.
  Recall first that the definition of the $L^1$-Kantorovitch distance (sometimes also called Wasserstein or transportation metric) between two random variables $X,Y$ taking values in a metric space $(S,d)$ is given by
  \begin{equation}\label{L1W}
    W_1(X,Y):= \inf \E[d(\hat X,\hat Y)]
  \end{equation}
  where the infimum is taken over all couplings $(\hat X,\hat Y)$ which are distributed marginally as $X$ and $Y$ respectively. Ollivier's definition of Ricci curvature of a Markov chain $(X_t, t \ge 0)$ on a metric space $(S,d)$ is as follows:

  \begin{definition}
    Let $t>0$. The curvature between two points $x, x'\in S$ with $x \neq x'$ is given by
    \begin{equation}\label{D:Ricci}
      \kappa_t (x, x') : = 1 - \frac{W_1 (X_t^x, X_t^{x'})}{d(x,x')}
    \end{equation}
    where $X_t^x$ and $X_t^{x'}$ denote Markov chains started from $x$ and $x'$ respectively. The curvature of $X$ is by definition equal to
    $$
    \kappa_t : = \inf_{x \neq x'} \kappa_t(x,x').
    $$
  \end{definition}

  In the terminology of Ollivier~\cite{ollivier_ricci}, this is in fact the curvature of the discrete-time random walk whose transition kernel is given by $m_x(\cdot) = \P(X_t = \cdot | X_0 = x)$. We refer the reader to \cite{ollivier_ricci} for an account of the elegant theory which can be developed using this notion of curvature, and point out that a number of classical properties of curvature generalise to this discrete setup.

  For our results it will turn out to be convenient to view the symmetric group
  as a metric space equipped with the metric $d$ which is the word metric induced by the set $T$ of transpositions (we will do so even when the random walk is not induced by $T$ but by a general conjugacy class $\Gamma$). That is, the distance $d(\sigma, \sigma')$ between $\sigma, \sigma' \in \Sn$ is the minimal number of transpositions one must apply to get from one element to the other (one can check that this number is independent of whether right-multiplications or left-multiplications are used).

  For simplicity we focus in this introduction on the case where the random walk is induced by the set of transpositions $T$. (A more general result will be stated later on the paper). For $c>0$ and $\sigma \neq \sigma'$, let
  \begin{equation}\label{eq:curvature_def}
    \kappa_c(\sigma,\sigma') = 1- \frac{W_1(X^{\sigma}_{\lfloor cn/2 \rfloor /},X^{\sigma'}_{\lfloor cn/2 \rfloor })}{d(\sigma,\sigma')}
  \end{equation}
  and define $\kappa_c(\sigma,\sigma)=1$. That is, $\kappa_c(\sigma, \sigma') = \kappa_{\lfloor cn/2\rfloor }(\sigma, \sigma')$ with our notation from \eqref{D:Ricci}. In particular, $\kappa_c$ depends on $n$ but this dependency does not appear explicitly in the notation.
  It is not hard to see that $\kappa_c(\sigma, \sigma') \ge 0$ (apply the same transpositions to both walks $X^\sigma$ and $X^{\sigma'}$).
  For parity reasons it is obvious that that $\kappa_c(\sigma, \sigma')=0$ if $\sigma$ and $\sigma'$ do not have the same signature. Thus we only consider the curvature between elements of even distance. For $c>0$ define
  $$
  \kappa_c = \inf \kappa_c(\sigma, \sigma'),
  $$
  where the infimum is taken over all $\sigma,\sigma'\in \Sn$ such that $d(\sigma,\sigma')$ is even.
  Our main result states that  $\kappa_c$ experiences a phase transition at $c=1$. More precisely, the curvature $\kappa_c$ is asymptotically zero for $c \le 1$ but for $c>1$ the curvature is strictly positive asymptotically. In order to state our result, we introduce the quantity $\theta(c)$, which is the largest solution in $[0,1]$ to the equation
  \begin{equation}\label{eq:theta}
    \theta(c)=
    1-e^{-c \theta(c)}.
  \end{equation}
  It is easy to see that $\theta(c) = 0$ for $c\le1$ and $\theta(c) >0$ for $c>1$. In fact, $\theta(c)$ is nothing else but the survival probability of a Galton-Watson tree with Poisson offspring distribution with mean $c$.

  \begin{thm}\label{thm:curv}
For any $c>0$, we have:
    \begin{equation}\label{eq:thm_upper}
\theta(c)^4 \le     \liminf_{n \rightarrow \infty} \kappa_c \le \limsup_{n \rightarrow \infty} \kappa_c  \leq \theta(c)^2
    \end{equation}
    In particular, $\lim_{n \to \infty} \kappa_c = 0$ if and only if $c \le 1$, while $\liminf_{n \to \infty} \kappa_c >0$ otherwise.
  \end{thm}

  A more general version of this theorem will be presented later on, which gives results for the curvature of a random walk induced by a general conjugacy class $\Gamma$. This will be stated as Theorem \ref{thm:curv2}.

  We believe that the upper bound is the sharp one here, and thus make the following conjecture.
  \begin{conj} For $c>0$,
    $$
    \lim_{n \to \infty} \kappa_c = \theta(c)^2.
    $$
  \end{conj}
  Of course the conjecture is already established for $c\le 1$ and so is only interesting for $c>1$.

\subsection{Relation to previous works on the geometry of random transpositions}
The transition described by Theorem \ref{thm:curv} says that the discrete Ricci curvature increases abruptly (asymptotically) from zero to a positive quantity as $c$ increases
past the critical value $c=1$, and so as we consider longer portions of the random walk.
It is related to a result proved by the first author in \cite{Berestycki_hyper}. There it was shown that the triangle formed by the identity and two independent samples $X_t$ and $X'_t$ from the random walk run for time $t=cn/4$, is thin (in the sense of Gromov hyperbolicity) if and only if $c<1$. Note that by reversibility, the path running from $X_t$ to $X'_t$ (via the identity) is a random walk run for time $cn/2$.
 In other words, the result from \cite{Berestycki_hyper} implies that the permutation group appears Gromov hyperbolic from the point of view of a random walker so long as it takes fewer than $cn/2$ steps with $c< 1$.

 Hence, in both Theorem \ref{thm:curv} and \cite{Berestycki_hyper}, there is a change of geometry (as perceived by a random walker) from low to high curvature
 after running for exactly $t = cn/2$ steps with $c=1$.
 At this point, we do not know of a formal way to relate these two observations, so they simply seem analogous. In a private conversation with the first author in 2005, Gromov had suggested that the hyperbolicity transition of \cite{Berestycki_hyper} could be translated more canonically into the language of Ricci curvature and was an effect of the global positive curvature of $\cS_n$ rather than a breakdown in hyperbolicity. In a sense, Theorem \ref{thm:curv} can be seen as a formalisation and justification of his prediction.

  \subsection{Relation to previous works on mixing times}

  Mixing times of Markov chains were initiated independently by Aldous~\cite{aldous_rapid_mixing} and by Diaconis and Shahshahani~\cite{diaconis_mixing}. In particular, as already mentioned, Diaconis and Shahshahani proved Theorem \ref{thm:mix} in the case where $\Gamma$ is the set $T$ of transpositions. Their proof relies on some deep connections with the representation theory of $\Sn$ and bounds on so-called character ratios.
  The conjecture about the general case appears to have first been made formally in print by Roichman~\cite{roichman_upper} but it has no doubt been asked privately before then. We shall see that that the lower bound $\tmix(\delta) \geq (1+ o(1))(1/|\Gamma|)n \log n$ is fairly straightforward (it is carried out in Appendix \ref{S:lb} and is as usual based on a coupon-collector type argument); the difficult part is the corresponding upper bound.

  Flatto, Odlyzko and Wales~\cite{flatto_group} built on the earlier work of Vershik and Kerov~\cite{vershik_aymptotic} to obtain that $\tmix(\delta) \leq (1/2 + o(1))n \log n$ when $|\Gamma|$ is bounded (as is noted in \cite[p.44-45]{diaconis_group}).  This was done using character ratios and this method was extended further by Roichman~\cite{roichman_upper,roichman_characters} to show an upper bound on $\tmix(\delta)$ which is sharp up to a constant when $|\Gamma | = o(n)$ (and in fact, more generally when $|\Gamma|$ is allowed to grow to infinity as fast as $(1-\delta)n$ for any $\delta \in (0,1)$). Again using character ratios Lulov and Pak~\cite{lulov_rapid} showed the cut-off phenomenon as well as $\tmix = (1/|\Gamma|)n \log n$ in the case when $|\Gamma| \geq n/2$. Roussel~\cite{roussel_phd, roussel_cutoff} obtains the correct value of the mixing time and establishes the cut-off phenomenon for the case $|\Gamma|\leq 6$.

  Finally, let us discuss two more recent papers to which this work is most closely related to.  Berestycki, Schramm and Zeitouni \cite{berestycki_k}, show using coupling arguments and a connection to coalescence-fragmentation processes that the cutoff phenomenon occurs at $\tmix = (1/k)n \log n$ in the case when $\Gamma$ consists only of cycles of length $k$ for any $k\geq 2$ fixed.

  Shortly after, Bormashenko \cite{bormashenko} devised a path coupling argument for the coagulation-fragmentation process associated to random transpositions to obtain a new proof of a slightly weaker version of the Diaconis--Shahshahani result: her argument implies that the mixing time of random transpositions is $O( n \log n)$ (unfortunately the implicit multiplicative constant is not sharp, so this is not sufficient to obtain cutoff). See also \cite{paulin} for another discussion of her results together with a reformulation in the language of coarse Ricci curvature. In a way her approach is very similar to ours, to the point that it can be considered a precursor to our work, since our method is also based on a certain path coupling for the coagulation-fragmentation process which exploits certain remarkable properties of Schramm's coupling \cite{schramm_transpo, berestycki_k}.

 \paragraph{Comparison with \cite{berestycki_k}.}The authors in Berestycki, Schramm and Zeitouni~\cite{berestycki_k} remark that their proof can be extended to cover the case when $\Gamma$ is a fixed conjugacy class and indicate that their methods can probably be pushed to cover the case when $|\Gamma|=o(n^{1/2})$, but it is clear that new ideas are needed if $|\Gamma|$ is larger. Indeed, their argument uses very delicate estimates about the behaviour of small cycles, together with a variant of a coupling due to Schramm~\cite{schramm_transpo} to deal with large cycles. The most technical part of their argument is to analyse the distribution of small cycles, using delicate couplings and carefully bounding the error made in these couplings.

  \medskip However, when $k = |\Gamma|$ is larger than $n^{1/2}$, we can no longer think of the points in the conjugacy class as being sampled independently (with replacement) from $\{1, \ldots, n\}$, by the birthday problem. This introduces many more ways in which errors in the above coupling arguments could occur. These seem quite hard to control, and hence new ideas are required for the general case.

  \medskip The proof in this paper relies on similar observations as \cite{berestycki_k}, and in particular the connection with coalescence-fragmentation process as well as Schramm's coupling argument play a crucial role. The key new idea however, is to try to prove mixing not just in the total variation sense but in the stronger sense of the $L^1$-Kantorovitch distance (Ricci curvature) and to estimate it at a time well before the mixing time, roughly $O(n/k)$ instead of $O(n(\log n) / k)$. This may seem counterintuitive initially, however studying the random walk at this time scale allows us to make precise comparisons between the random walk and an associated random graph process. It turns out the random graph at these time scales can be described rather precisely. Furthermore, due to the contraction properties of the Kantorovich distance, somehow (and rather miraculously, we find), the estimate we obtain can be bootstrapped with sufficient precision to yield mixing exactly at the time $\tmix = (1/k) n \log n$.

  \medskip In particular, since the heart of the proof consists in studying the situation at a time well before mixing, and purely to take advantage of the giant component at such times, we \emph{never} have to study the distribution of small cycles. This is really quite surprising, given that the small cycles (in particular, the fixed points) are responsible for the occurrence of the cutoff at time $\tmix$.

\subsection{Organisation of the paper.} We stress that compared to \cite{berestycki_k}, the main arguments are quite elementary. The heart of the proof is contained in Sections \ref{sec:lower_bound} and \ref{Sec:mixing}. Readers who are familiar with \cite{berestycki_k} are encouraged to concentrate on these two short sections.

The paper is organised as follows. In Section \ref{Sec:mixing} we state and discuss Theorem \ref{thm:curv2}, which is a general curvature theorem (of which Theorem \ref{thm:curv} is the prototype). We also discuss why this implies the main theorem (Theorem \ref{thm:mix}). In Section \ref{SS:hyper} we study the associated random hypergraph process. The main result in that section is Theorem \ref{lemma:hyper_k_cycle}, which proves the existence and uniqueness of the giant component. Curiously this is the most technical aspect of the paper, and really the only place where the myriad of  ways in which the conjugacy class $\Gamma$ might be really big plays a role and needs to be controlled. Section \ref{sec:curv} contains a proof of the main curvature theorem (Theorem \ref{thm:curv2}), starting with the easy upper bound on curvature (Section \ref{sec:upper_bound}) and following up with the slightly more complex lower bound (Section \ref{sec:lower_bound}), which really is the heart of the proof. The two appendices contain respectively a proof of the lower bound on the mixing time (certainly known in the folklore, essentially a version of the coupon collector lemma); and an adaptation of Schramm's argument \cite{schramm_transpo} for the Poisson--Dirichlet structure of cycles inside the giant component, which is needed in the proof.

\medskip\noindent  \textbf{Acknowledgements.} We thank Yuval Peres and Spencer Hughes for useful discussions on discrete Ricci curvature.

  \section{Curvature and mixing}\label{Sec:mixing}

  \subsection{Curvature theorem}

As discussed above, the lower bound \eqref{cutoff} is relatively easy and is probably known in the folklore; we give a proof in Appendix \ref{S:lb}. We now start the proof of the main results of this paper, which is the upper bound (the right hand side) of \eqref{cutoff}. In this section, we first state the more general version of Theorem \ref{thm:curv} discussed in the introduction, and we will then show how this implies the desired result for the upper bound on $\tmix(\delta)$. To begin, we define the cycle structure $(k_2,k_3, \dots)$ of $\Gamma$ to be a vector such that for each $j \geq 2$, there are $k_j$ cycles of length $j$ in the cycle decomposition of any $\gamma \in \Gamma$ (note that this does not depend on $\tau\in \Gamma$). Then $k_j=0$ for all $j>n$ and we have that $k :=|\Gamma| = \sum_{j=2}^\infty j k_j$.

  In the case for the transposition random walk the quantity $\theta(c)$ which appears in the bounds is the survival probability of a Galton-Watson process with offspring distribution given by a Poisson random variable with mean $c$. Our first task is to generalise $\theta(c)$. We do so via a fixed point equation, which is more complex here.
   Define
  $$
  \alpha_j = \frac{j k_j}{k} ,
   $$
   and note that $\alpha_j \in [0,1]$ ($\alpha_j$ is the proportion of the mass in cycles of size $j$ for any $\gamma \in \Gamma$). Thus $(\alpha_j)_{j \geq 2}$ is compact in the product topology (the topology of pointwise convergence).
  Suppose that the limit
  \begin{equation}\label{eq:k'}\tag{C}
    (\bar \alpha_2,\bar \alpha_3,\dots):=\lim_{n\to\infty}\left( \alpha_2, \alpha_3 , \dots \right)
  \end{equation}
  exists, where the limit is taken to be pointwise. It follows that for each $j \geq 2$, $\bar \alpha_j \in [0,1]$ and $\sum_{j=2}^\infty \bar \alpha_j \leq 1$ by Fatou's lemma. Note that the sum is strictly less than 1 when a positive fraction of the mass of conjugacy class $\Gamma$ comes from cycles whose size tends to $\infty$. This will be an important distinction in what follows. For $x \in [0,1]$ and $c>0$ define
  \begin{equation}\label{eq:Psi}
    \psi(x,c)=\exp\Big\{-c \big(1-\sum_{j=2}^\infty \bar \alpha_j (1-x)^{j-1}\big)\Big\}.
  \end{equation}
  Note that for each $c>0$, $x \mapsto \psi(x,c)$ is convex on $[0,1]$. Moreover, the function $ x \mapsto \psi(1-x,c)$ is the generating function of a random variable whose law depends on $c$ and is degenerate if $\sum_{j \geq 2} \bar \alpha_j <1$. Note that in the case $\Gamma = T$ of transpositions, $\psi(x,c) = e^{-cx}$ so that random variable is simply Poisson $(c)$.

  \begin{lemma}\label{lemma:root_psi}
    Define
    \begin{equation}\label{eq:c_Gamma}
      c_\Gamma:= \begin{cases}
	\left(\sum_{j=2}^\infty (j-1)\bar \alpha_j \right)^{-1} & \text{ if } \sum_{j=2}^\infty \bar \alpha_j=1\\
	\quad 0 & \text{ if } \sum_{j=2}^\infty \bar \alpha_j<1.
      \end{cases}
    \end{equation}
    Then for $c>c_\Gamma$ there exists a unique $\theta(c) \in (0,1)$ such that
    \[
      \theta(c)=1-\psi(\theta(c),c).
    \]
    For $c>c_\Gamma$, $c \mapsto \theta(c)$ is increasing, continuous and differentiable. Further $\lim_{c \downarrow c_\Gamma}\theta(c)=0$ and $\lim_{c \uparrow \infty}\theta(c)=1$.
  \end{lemma}
  \begin{proof}
    For $x \in [0,1]$ and $c>0$ define $f_c(x):=1-\psi(x,c)-x$. There are two cases to consider. First suppose that $z=\sum_{j=2}^\infty \bar \alpha_j<1$. Then we have that
    \[
      f_c(0)=1-e^{-c(1-z)}>0 \qquad \text{ and } \qquad f_c(1)=-e^{-c}<0.
    \]
    As $x \mapsto f_c(x)$ is concave on $[0,1]$ it follows that there exists a unique $\theta(c) \in (0,1)$ such that $f_c(\theta(c))=0$.

    Next suppose that $\sum_{j=2}^\infty \bar \alpha_j=1$, then
    \[
      f_c(0)=0 \qquad \text{ and } \qquad f_c(1)=-e^{-c}<0
    \]
    Moreover we have that
    \[
      \frac{d}{dx}f_c(x)|_{x=0}=c \tsum_{j=2}^\infty (j-1)\bar \alpha_j - 1.
    \]
    Hence for $c> c_\Gamma$ we have that $\frac{d}{dx}f_c(x)|_{x=0} >0$ and again by concavity it follows that there exists a unique $\theta(c) \in (0,1)$ such that $f_c(\theta(c))=0$.

    For the rest of the statements suppose that $c>c_\Gamma$. The fact that $c \mapsto \theta(c)$ is increasing follows from the definition of $\psi(x,c)$ and the fact that $\theta(c)=\psi(\theta(c),c)$. Continuity and differentiability for $c> c_\Gamma$ is a straightforward application of the inverse function theorem.

    Notice that $\theta(c) \in [0,1]$ and is monotone, hence $\theta(c)$ converges as $c \downarrow c_\Gamma$ to a limit $L$. Then it follows that $L$ solves the equation $L=1-\psi(L,c_\Gamma)$. This equation has only a zero solution and thus $L=0$ and hence $\lim_{c \downarrow c_\Gamma}\theta(c)=0$. The limit as $c \uparrow \infty$ follows from a similar argument.
  \end{proof}

 \begin{rem}
  In the case when $\Gamma=T$ is the set of transpositions we have that $k'_2 = 1$ and $\bar \alpha_j = 0$ for $j\ge 3$, hence $\psi(x,c)=e^{-cx}$ and thus the definition of $\theta(c)$ above agrees with the definition given in the introduction.
\end{rem}

  Having introduced $\theta(c)$ we now introduce the notion of Ricci curvature we will use in the general case. For $c>0$ and $\sigma\neq \sigma'$, let
  \begin{equation}\label{eq:curvature_def2}
    \kappa_c(\sigma,\sigma') = 1- \frac{W_1(X^{\sigma}_{\lfloor cn/k \rfloor },X^{\sigma'}_{\lfloor cn/k \rfloor })}{d(\sigma,\sigma')}
  \end{equation}
  where $d$ is the graph distance associated with transpositions (\emph{even} in the case $\Gamma \neq T$). Define $\kappa_c(\sigma,\sigma)=1$. Then let
  $$
  \kappa_c = \inf \kappa_c(\sigma, \sigma'),
  $$
  where the infimum is taken over all $\sigma,\sigma'\in \Sn$ such that $d(\sigma,\sigma')$ is even.
  That is, $\kappa_c(\sigma, \sigma') = \kappa_{\lfloor cn/k\rfloor }(\sigma, \sigma')$ with our notation from \eqref{D:Ricci}. We now state a more general form of Theorem \ref{thm:curv} which in particular covers the case of Theorem \ref{thm:curv}.

  \begin{thm}\label{thm:curv2}
    Let $\Gamma\subset \Sn$ be a conjugacy class such that $k = |\Gamma|=o(n)$ and the convergence in~\eqref{eq:k'} holds. Recall the definition of $c_\Gamma$ from \eqref{eq:c_Gamma}. Then for $c\leq c_\Gamma$,
    \begin{equation}\label{E:subcrit1}
      \lim_{n \to \infty} \kappa_c = 0.
    \end{equation}
    On the other hand, for $c>c_\Gamma$
    \begin{equation}\label{eq:thm_lower2}
     \theta(c)^4 \le \liminf_{n \rightarrow \infty} \kappa_c  \le  \limsup_{n \rightarrow \infty} \kappa_c \le \theta(c)^2
    \end{equation}
		where $\theta(c)$ is the unique solution in $(0,1)$ of
    \begin{equation}\label{eq:theta2}
      \theta(c)=1-\psi(\theta(c),c).
    \end{equation}
    where $\psi$ is given by \eqref{eq:Psi}.
  \end{thm}

  \subsection{Curvature implies mixing}

  We now show how Theorem \ref{thm:curv2} implies the second limit in Theorem \ref{thm:mix}.
  First suppose that $\Gamma=\Gamma(n)$ is a sequence of conjugacy classes for which the limit~\eqref{eq:k'} holds and $|\Gamma|=o(n)$.
  Again fix $\epsilon>0$ and define $t= (1+2\epsilon)(1/k)n \log n$ and let $t' = \lfloor (1+\epsilon)(1/k)n \log n\rfloor$ where $k=|\Gamma|$. We are left to prove that $d_{TV}(t) \rightarrow 0$ as $n \rightarrow \infty$. For $s \geq 0$ let
  $$
  \bar d_{TV}(s) : = \sup_{\sigma,\sigma'} \| X^\sigma_s - X^{\sigma'}_s \|_{TV},
  $$
  where the sup is taken over all permutations at even distances. We first claim that it suffices to prove that
  \begin{equation}\label{dbar}
    \bar d_{TV}(t') \to 0 \text{ as } n \rightarrow \infty.
  \end{equation}
  Indeed, assume that $\bar d_{TV}(t') \to 0$ as $n \rightarrow \infty$. Then there are two cases to consider. Assume that $\Gamma \subset \An$. Then $X_s \in \An$ for all $s\ge 1$ and $\mu$ is uniform on $\An$. Then by Lemma 4.11 in \cite{peres_mixing_book},
  $$
  \sup_{\sigma \in \An} \| X^\sigma_{t'} - \mu \|_{TV} \le 2 \bar d_{TV}(t').
  $$
  Hence Theorem \ref{thm:mix} (or more precisely the second limit in that theorem) follows from \eqref{dbar} in this case. In the second case, $\Gamma \subset \An^c$. In this case $X_s \in\An$ for $s$ even, and $X_s \in \An^c$ for $s$ odd. Using the same lemma, we deduce that if $s \ge t'$ is even,
  $$
  \| X^{\id}_s - \mu_1 \|_{TV}  \le 2 \bar d_{TV}(s)
  $$
  where $\mu_1$ is uniform on $\An$. However, if $s\ge t'$ is odd,
  $$
  \| X^{\id}_s - \mu_2 \|_{TV} \le 2 \bar d_{TV}(s)
  $$
  where this time $\mu_2$ is uniform on $\An^c$. Let $N=(N_s:s \geq 0)$ be the Poisson clock of the random walk $Y$. Then $\P(N_s \text{ even} ) \to 1/2$ as $s\to \infty$, $\mu = (1/2)(\mu_1+ \mu_2)$, and $\P(N_t \ge t') \to 1$ as $n \to \infty$. Thus we deduce that
  $$
  \| Y^{\id}_t - \mu \|_{TV}\to 0.
  $$
  Again, the second limit in Theorem \ref{thm:mix} follows. Hence it suffices to prove \eqref{dbar}.

  Note that for any two random variables $X,Y$ on a metric space $(S,d)$ we have the obvious inequality $\|X - Y\|_{TV} \leq W_1(X,Y)$ provided that $x\neq y $ implies $d(x,y) \ge 1$ on $S$. This is in particular the case when $S= \Sn$ and $d$ is the word metric induced by the set $T$ of transpositions.
  In other words it suffices to prove mixing in the $L^1$-Kantorovitch distance.

Note that by definition of $\kappa_c$,  if $\sigma$, $\sigma'$ are at an even distance then $$W_1(X^{\sigma}_{\lfloor cn/k \rfloor },X^{\sigma'}_{\lfloor cn/k \rfloor })  \le (1- \kappa_c) d( \sigma, \sigma'),
$$
so that, iterating as in Corollary 21 of \cite{ollivier_ricci} (and noting that the distance between $X^{\sigma}_{\lfloor cn/k \rfloor }$ and $X^{\sigma'}_{\lfloor cn/k \rfloor }$ is again even),  we have for each $s \geq 1$,
  \begin{equation}\label{eq:s_kappa_bound}
		\sup_{d(\sigma, \sigma') \text{ even}}W_1(X^{\sigma}_{s\lfloor cn/k \rfloor },X^{\sigma'}_{s \lfloor cn/k \rfloor })  \leq (1-\kappa_c)^s  \sup_{d(\sigma, \sigma') \text{ even}}d(\sigma, \sigma') \le n (1- \kappa_c)^s
  \end{equation}
  since the diameter of $\Sn$ is equal to $n-1$.
  Solving
  $
    n (1-\kappa_c)^s \leq \delta
    $
  we get that
  \begin{equation}\label{eq:s_lower_bound}
    s \geq \frac{\log n -\log \delta }{-\log(1-\kappa_c)}
  \end{equation}
  Thus if $u = scn/k \ge s \lfloor cn/k\rfloor $, it suffices that
  \begin{equation}\label{cond}
    u \ge  \frac1k \frac{c}{- \log (1- \kappa_c)} n( \log n - \log \delta).
  \end{equation}
  Now, Theorem \ref{thm:curv2} gives
  \[
    \liminf_{n \rightarrow \infty}-\log(1-\kappa_c) \geq -\log(1-\theta(c)^4).
  \]

  \begin{lemma}\label{lemma:log_theta}
    We have that
    \[
      \lim_{c \rightarrow \infty} \frac{c}{\log(1-\theta(c)^4)} = -1.
    \]
  \end{lemma}
  \begin{proof}
    Using L'Hopital's rule twice we have that
    \[
      \lim_{\theta \uparrow 1} \frac{\log(1-\theta)}{\log(1-\theta^4)}=\lim_{\theta \uparrow 1}\frac{1-\theta^4}{(1-\theta)4\theta^3}=1.
    \]
    Next we have that $\lim_{c \rightarrow \infty} \theta(c)=1$ and hence
    \begin{align*}
      \lim_{c \rightarrow\infty} \frac{c}{\log(1-\theta(c)^4)}&= \lim_{c\rightarrow\infty}\frac{c}{\log(1-\theta(c))}=\lim_{c\rightarrow\infty}\frac{c}{\log \psi(\theta(c),c)}\\
      &=\lim_{c\rightarrow\infty}-\frac{1}{1-\sum_{j=2}^\infty \bar \alpha_j(1-\theta(c))^{j-1}}=-1.
    \end{align*}
  \end{proof}

  Consequently we have that for $u\ge t'=\lfloor (1+\epsilon)(1/k)n \log n\rfloor$ $u$ satisfies \eqref{cond} for some sufficiently large $c>c_\Gamma$.
  Hence $\limsup_{n \rightarrow \infty} \bar d_{TV}(t')\to 0$ and thus \eqref{dbar} holds, which shows Theorem~\ref{thm:mix} for conjugacy classes such that the limit in~\eqref{eq:k'} exists and $|\Gamma|=o(n)$.

\medskip   Now suppose that $\Gamma$ is a conjugacy class such that $|\Gamma|=o(n)$.
  Let $t' = \lfloor (1+\epsilon)(1/|\Gamma|)n \log n\rfloor$ and notice that $d_{TV}(t')$ is bounded.
  Along any subsequence $\{n_i\}_{i \geq 1}$ such that $\lim_{n_i\to\infty}d_{TV}(t')$ exists, we can extract a further sub-sequence $\{n_{i_j}\}_{j \geq 1}$ such that~\eqref{eq:k'} holds since $(\alpha_j)_{j \ge 2} \in [0,1]^\infty$ which is compact under the product topology. Then we see that $\lim_{n_{i_j}\to\infty}d_{TV}(t')=0$ and consequently $\lim_{n_i\to\infty}d_{TV}(t')=0$.
  Since $d_{TV}(t')$ is bounded and converges to $0$ along any convergent subsequence, we conclude that $\lim_{n\to\infty}d_{TV}(t')=0$, thus concluding the proof.

  \subsection{Stochastic commutativity}

  To conclude this section on curvature, we state a simple but useful lemma. Roughly, this says that the random walk is ``stochastically commutative''. This can be used to show that the $L^1$-Kantorovitch distance is decreasing under the application of the heat kernel. In other words, initial discrepancies for the Kantorovitch metric between two permutations are only smoothed out by the application of random walk.

  \begin{lemma}\label{lemma:transpo_end}
    Let $\sigma$ be a random permutation with distribution invariant by conjugacy. Let $\sigma_0$ be a fixed permutation. Then $\sigma_0 \circ \sigma$ has the same distribution as $\sigma\circ \sigma_0 $.
  \end{lemma}
  \begin{proof}
    Define $\sigma' = \sigma_0 \circ \sigma \circ \sigma_0^{-1}$. Then since $\sigma$ is invariant under conjugacy, the law of $\sigma'$ is the same as the law of $\sigma$. Furthermore, we have $ \sigma_0 \circ \sigma = \sigma'  \circ \sigma_0$ so the result is proved.
  \end{proof}

%
%
  This lemma will be used repeatedly in our proof, as it allows us to concentrate on events of high probability for our coupling.

  \section{Preliminaries on random hypergraphs}

  For the proof of Theorem \ref{thm:mix} we rely on properties of certain random hypergraph processes. The reader who is only interested in a first instance in the case of random transpositions, and is familiar with Erd\H{o}s--Renyi random graphs and with the result of Schramm~\cite{schramm_transpo} may safely skip this section.

  \subsection{Hypergraphs}
  \label{SS:hyper}
	In this section we present some preliminaries which will be used in the proof of Theorem \ref{thm:curv2}. Throughout we let $\Gamma \subset \Sn$ be a conjugacy class and let $(k_2,k_3,\dots)$ denote the cycle structure of $\Gamma$. Thus $\Gamma$ consists of permutations such that in their cycle decomposition they have $k_2$ many transpositions, $k_3$ many $3$-cycles and so on. Note that we have suppressed the dependence of $\Gamma$ and $(k_2,k_3,\dots)$ on $n$. We assume that \eqref{eq:k'} is satisfied so that for each $j\geq 2$, $j k_j/|\Gamma| \rightarrow \bar \alpha_j$ as $n \rightarrow \infty$. We also let $k=|\Gamma|$ so that $k=\sum_{j \geq 2} j k_j$, as usual.


  \begin{definition}
    A hypergraph $H=(V,E)$ is given by a set $V$ of vertices and $E \subset \mathcal{P}(V)$ of edges, where $\mathcal{P}(V)$ denotes the set of all subsets of $V$. An element $e \in E$ is called a hyperedge and we call it a $j$-hyperedge if $|e|=j$.
  \end{definition}

  Consider the random walk $X= (X_t: t = 0,1 \ldots)$ on $\Sn$ where $X_t = X^{\id}_t$ with our notations from the introduction. Hence
  $$
  X_t = \gamma_1 \circ \ldots \circ \gamma_t
  $$
  where the sequence $(\gamma_i)_{i\ge 1}$ is i.i.d. uniform on $\Gamma$. A given step of the random walk, say $\gamma_s$, can be broken down into cycles, say $\gamma_{s,1} \circ \ldots \gamma_{s,r}$ where $r = \sum_j k_j$. We will say that a given cyclic permutation $\gamma$ has been applied to $X$ before time $t$ if $\gamma = \gamma_{s,i}$ for some $s\le t$ and $1 \le i \le r$.

  To $X$ we associate a certain hypergraph process $H=(H_t:t=0,1, \ldots)$ defined as follows. For $t =0,1 ,\ldots$, $H_t$ is a hypergraph on $\{1,\dots,n\}$ where a hyperedge $\{x_1,\dots,x_j\}$ is present if and only if a cyclic permutation consisting of the points $x_1,\dots,x_j$ in some arbitrary order has been applied to the random walk $X$ prior to time $t$ as part of one of the $\gamma_i$'s for some $i \le t$. Thus at every step, we add to $H_t$ $k_j$ hyperedeges of size $j$ sampled uniformly at random without replacement, and these edges are independent from step to step. However, note that the presence of hyperedges themselves are not in general independent.

  \subsection{Giant component of the hypergraph}

  In the case $\Gamma=T$, the set of transpositions, the hypergraph $H_s$ is a realisation of an Erd\H{o}s-Renyi graph. Analogous to Erd\H{o}s-Renyi graphs, we first present a result about the size of the components of the hypergraph process $H=(H_t:t= 0,1,\dots)$ (where by size, we mean the number of vertices in this component). For the next result recall the definition of $\psi(x,c)$ in \eqref{eq:Psi}. Recall that for $c> c_\Gamma$, where $c_\Gamma$ is given by \eqref{eq:c_Gamma}, there exists a unique root $\theta(c)\in (0,1)$ of the equation $\theta(c)=1-\psi(\theta(c),c)$.

  \begin{thm}\label{lemma:hyper_k_cycle}
    Consider the random hypergraph $H_s$ and suppose that $s=s(n)$ is such that $sk/n \rightarrow c$ as $n \rightarrow \infty$ for some $c> c_\Gamma$.  Then there is a universal constant $D>0$ such that with probability tending to one all components but the largest have size at most $ D n^{2/3} (\log (n))^3$.  Furthermore, the size of the largest component, normalised by $n$, converges to $\theta(c)$ in probability as $n\to \infty$.
  \end{thm}

Of course, this is the standard Erd\H{o}s--Renyi theorem in the case where $\Gamma = T$ is the set of transpositions. See for instance \cite{durrett_book}, in particular Theorem 2.3.2 for a proof. In the case of $k$-cycles with $k$ fixed and finite, this is the case of random regular hypergraphs analysed by Karo{\'n}ski and {\L}uczak \cite{KaronskiLuczak}. For the slightly more general case of bounded conjugacy classes, this was proved by Berestycki \cite{Berestycki_cycles}.

\paragraph{Discussion.}
Note that the behaviour of $H_s$ in Theorem \ref{lemma:hyper_k_cycle} can deviate markedly from that of Erd\H{o}s--Renyi graphs. The most obvious difference is that $H_s$ can contain mesoscopic components, something which has of course negligible probability for Erd\H{o}s-Renyi graphs. For example, suppose $\Gamma$ consists of $n^{1/2}$ transpositions and one cycle of length $n^{1/3}$. Then the giant component appears at time $n^{1/2}/2$ with a phase transition (i.e., $c_\Gamma>0$, because in this case $\sum \bar \alpha_j =1$, as most of the mass comes from microscopic cycles). Yet  even at the first step there is a component of size $n^{1/3}$. Nevertheless we will see that once there is a giant component there is a limit to how big can the nongiant component be (we show this is less that $O(n^{2/3})$ up to logarithmic terms; this is certainly not optimal).

From a technical point of view this has nontrivial consequences, as proofs of the existence of a giant component are usually based on the dichotomy between microscopic components and giant components. Furthermore, when the conjugacy class is large and consists of many small or mesoscopic cycles, the hyperedges have a strong dependence, which makes the proof very delicate.

\medskip In effect, perhaps surprisingly this will be the only place of the proof where all the possible ways in which the conjugacy class $\Gamma$ might be big (potentially of size very close to $n$), needs to be handled. The difficulty of the proof below is to find an argument which works no matter how $\Gamma$ is made up, so long as $k = |\Gamma| = o(n)$. This is of course also the problem in the original question of studying the mixing time of the random walk induced by $\Gamma$. However, what we have gained here compared to this original question, is the monotonicity of component sizes when hyperedges are added to $H_s$.



 \paragraph{Preliminaries: exploration.} Suppose that $s=s(n)$ is such that $sk/n \rightarrow c$ for some $c> 0$ as $n \rightarrow \infty$ for some $c\geq 0$.
We reveal the vertices of the component containing a fixed vertex $v \in \{1,\dots,n\}$ using breadth-first search exploration, as follows. There are three states that each vertex can be: unexplored, removed or active. Initially $v$ is active and all the other vertices are unexplored. At each step of the iteration we select an active vertex $w$ according to some prescribed rule among the active vertices at this stage (say with the smallest label). The vertex $w$ becomes removed and every unexplored vertex which is joined to $w$ by a hyperedge becomes active. We repeat this exploration procedure until there are no more active vertices.
At stage $i = 0,1, \ldots$ of this exploration process, we let $A_i$, $R_i$ and $U_i$ denote the set of active, removed and unexplored vertices respectively. Thus initially $A_0=\{v\}$, $U_0=\{1,\dots,n\}\backslash \{v\}$ and $R_0=\emptyset$. We will let $a_i = |A_i|, u_i = |U_i|, r_i = |R_i|$.

For $t =1,\dots,s$ we call the hyperedges which are associated with the permutation $\gamma_t$ the $t$-th packet of hyperedges. Thus note that each packet consists of $k_j$ hyperedges of size $j$, $j \ge 2$, which are sampled uniformly at random without replacement from $\{1, \ldots,n\}$. In particular, within a given packet, hyperedges are not independent. However, crucially, hyperedges from different packets are independent. We will need to keep track of the hyperedges we reveal and where they ``came from" (i.e., which packet they were part of), in order to deal with these dependencies. More precisely, as we explore the hypergraph $H_s$, we discover various hyperedges of various sizes in $H_s$ and this may affect the likelihood of other types of hyperedges in subsequent steps of the exploration process. To account for this, we introduce for $t=1,\dots,s$ and for $j \geq 2$, the random subset of $\{1, \ldots, n\}$, $Y^{(t)}_{j} (i)$, which is defined to be the hyperedges of size $j$ in the $t$-th packet that were revealed in the exploration process prior to step $i$. We let $y_j^{(t)}(i) = |Y_j^{(t)}(i)|$ denote the number of such hyperedges.

\paragraph{Additional notations.} Let $i \ge 0$ and let $\cH_i$ denote the filtration generated by the exploration process up to stage $i$, including the information of the number of hyperedges of each size in each packet that were revealed up to step $i$ of the exploration process. That is,
$$
\cH_i = \sigma(A_1, \ldots, A_i, Y^{(t)}_j(1), \ldots, Y_j^{(t)}(i): 1\le t \le s, j \ge 2).
$$
Our first goal will be to give uniform stochastic bounds on the distribution of $a_{i+1} - a_i$, so long as $i$ is not too large. We will thus fix $i$ (a step in the exploration process) and in order to ease notations we will often suppress the dependence on $i$, in $Y^{(t)}_j(i)$: we will thus simply write $Y^{(t)}_j$ and $y_j^{(t)}$.
Note that by definition, for each $t=1,\dots,s$ and $j \geq 2$, $Y^{(t)}_j \leq k_j$ and
	\begin{equation}\label{eq:sum_Y}
		\tsum_{t=1}^s\sum_{j \geq 2} j y^{(t)}_j  \ge n- u_i= a_i + i,
	\end{equation}
{where the right hand side counts the total number of vertices explored by stage $i$, while the left hand side counts the sum of the sizes of all hyperedges revealed by stage $i$, so the $\ge$ sign accounts for possible intersections between the hyperedges.}

Let  $w$ be the vertex being explored for stage $i+1$. For $t = 1, \ldots, s$ let $M_t$ be the indicator that $w$ is part of an (unrevealed) hyperedge in the $t$-th packet. Thus, $(M_t)_{1\le t \le s}$ are independent conditionally given $\cH_i$, and
 \begin{equation}\label{Mt}
 \P( M_t = 1 | \cH_i) = \tsum_{j \geq 2}\tfrac{j(k_j-y^{(t)}_j)}{|U_{i}|},
\end{equation}
since $k_j - y_j^{(t)}$ counts the number of hyperedges of size $j$ still unrevealed in the $t$-th packet. If $w$ is part of a hyperedge in the $t$-th packet, let $V_t$ be the size of the (unique) hyperedge of that packet containing it.
Then
\begin{equation}\label{Vt}
		\P(V_t=j | \cH_i, M_t = 1)=
			\tfrac{j(k_j-y^{(t)}_j)}{\sum_{m \geq 2}m(k_m-y^{(t)}_m)}
\end{equation}
Note that when $M_t=1$ it implies that the denominator above is non-zero and thus \eqref{Vt} is well defined. When $M_t = 0$ we simply put $V_t = 1$ by convention. Then we have the following almost sure inequality:
\begin{equation}\label{eq:new_active}
		a_{i+1} - a_i \le -1+\tsum_{t=1}^s M_t(V_t-1).
\end{equation}
This would be an equality if it were not for possible self-intersections, as hyperedges connected to $w$ coming from different packets may share several vertices in common. In order to get a bound in the other direction, we simply truncate the $a_{i+1} - a_i$ at $n^{1/4}$. Let $I_i$ be the indicator that among the first $n^{1/4}$  {vertices to which $w$ is connected, no self-intersection or intersection with the past occurs.} Note that $\E(I_i) \ge p_n = 1- n^{-1/2}$, by straightforward bounds on the birthday problem. We then have
\begin{equation}\label{eq:new_active_lb}
(a_{i+1} - a_i)  \wedge n^{1/4} \ge -1 + I_i  \left (\tsum_{t=1}^s M_t ( V_t - 1) \wedge n^{1/4} \right).
\end{equation}

\paragraph{Organisation of proof of Theorem \ref{lemma:hyper_k_cycle}.} We will stop the exploration process once we have discovered enough vertices, or if the active set dies out, whichever comes first. We aim to show that starting from a given vertex $v$, with probability approximately $\theta(c)$ the cluster of $v$ contains about order $n$ vertices. However, we proceed in stages as different arguments are needed in order to reach so many vertices. In Step 1, we first show that the cluster contains about $(\log n)^2$ vertices with probability approximately $\theta(c)$. Then in Step 2, given that the exploration of the cluster has discovered $(\log n)^2$ vertices, we show that with high probability the exploration will in fact discover $n^{2/3}$ vertices. Finally, in Step 3 we show using the sprinkling technique that any two clusters that reach a size of about $n^{2/3}$ can be connected using only very few additional edges, which implies the result.

\medskip \noindent \textbf{Main quantitative lemma.} We define
\begin{align}
T^\downarrow& :=\inf\{i \geq 1: a_i =0\} \label{Tdown} \\
T^\uparrow &:= \inf\{i \geq 1:  a_i > n^{2/3} \} \label{Tup}
\end{align}
We set $T = T^\uparrow \wedge T^\downarrow$. 
Hence our first goal (which we will show at the end of Step 2) will be to show that $T = T^\uparrow$ with probability $\theta(c)$: in fact we will show that $T_\downarrow$ occurs before $T^\uparrow$ or $n^{2/3}$ with probability approximately $1- \theta(c)$. Either way, this means that the component is greater than $n^{2/3}$ with probability approximately $\theta(c)$. To do this we need to study the distribution of of $a_{i+1} - a_i$; the next lemma shows that these random variables converge in distribution to a sequence of i.i.d. (possibly degenerate) random variables, uniformly for $i <T$: the limit is improper if $\sum_j \bar \alpha_j < 1$.

Equivalently, the active process $|A_i|$ converges (at least for finite dimensional marginals) to the exploration process of a Galton--Watson tree whose offspring distribution is given by the limit of $a_{i+1} - a_i + 1$ and thus has a moment generating function given by $\psi(1-\cdot, c)$.

\medskip It is perhaps surprising that the lemma below is sufficient for the proof of Theorem \ref{lemma:hyper_k_cycle}: the lemma below essentially only records whether a cycle is microscopic (finite) or ``more than microscopic"; in particular, whether the mass of $\Gamma$ comes from many small mesoscopic or fewer big cycles makes no difference.

\begin{lemma}\label{lemma:A_mgf} For each $q_0 \in [0,1)$, there exists some deterministic function $w:\N \rightarrow \R$ such that $w(n) \to 0$ as $n \rightarrow \infty$ with the following property:
		\[
			\sup_{ i \ge 1 } \sup_{q \in  [0,q_0] }\Big|\E[q^{a_{i+1} - a_i} | \cH_{i}] - \tfrac{\psi(1-q,c)}{q}\Big| 1_{\{T > i \}}\leq w(n)
		\]
		almost surely.

	\end{lemma}

\begin{proof}
Suppose $T>i$.
From \eqref{eq:new_active} we have that
	\begin{align*}
		 q \E[q^{a_{i+1} - a_i}| \cH_{i}] & \ge  \E[ q^{\sum_{t=1}^s M_t(V_t-1)}| \cH_i]  = \tprod_{t=1}^s \E(q^{M_t (V_t -1)} | \cH_i)\\
		 & = \tprod_{t=1}^s \Big[1 - \P(M_t =1 |\cH_i) ( 1- \E( q^{V_t - 1} | \cH_i, M_t = 1) ) \Big].
	\end{align*}
Recall from \eqref{Vt} that
\begin{align*}
\E( q^{V_t - 1} | \cH_i, M_t = 1) 	&= \tsum_{j \geq 2} q^{j-1} \tfrac{j(k_j-y^{(t)}_j)}{\sum_{m \geq 2}m(k_m-y^{(t)}_m)} \ge  \tsum_{j \geq 2} q^{j-1} \tfrac{j(k_j-y^{(t)}_j)}{k}
\end{align*}
and from \eqref{Mt} that
$$
\P(M_t =1 |\cH_i) = \tsum_{m \geq 2}\tfrac{m(k_m-y^{(t)}_m)}{|U_{i}|} \le \tsum_{m \geq 2}\tfrac{mk_m}{n-2n^{2/3}} \le \tfrac{k}{n} (1+ 3 n^{-1/3})
$$
{by definition of $T^\uparrow$.}
 Therefore, using $ 1- x \ge e^{-x - x^2 }$ for all $x$ sufficiently small,
\begin{align}
 q \E[q^{a_{i+1} - a_i}| \cH_{i}] 	& \geq \tprod_{t=1}^s\Big[1- \tfrac{k}{n} (1+ 3 n^{-1/3}) \Big(1- \tsum_{j \geq 2} q^{j-1} \tfrac{j(k_j-y^{(t)}_j)}{k}\Big)\Big]  \label{offspring} \\
		&\geq \tprod_{t=1}^s\Big\{1-  \tfrac{k}{n}\Big( 1 + 3 n^{-1/3}- \tsum_{j \geq 2} q^{j-1}   \alpha_j \Big)-\tfrac{k}{n}{(1+ 3n^{-1/3})} \tsum_{j \geq 2} q^{j-1} \tfrac{j y^{(t)}_j}{k} \Big\} \nonumber\\
		& \geq \tprod_{t=1}^s\Big\{ 1- \tfrac{k}{n}\Big( 1- \tsum_{j \geq 2} q^{j-1} \alpha_j \Big)- 3n^{-1/3} \tfrac{k}n  - \tfrac1n {(1+ 3n^{-1/3})} \tsum_{j \ge 2} j y^{(t)}_j \Big\}  \nonumber\\
		& \geq \exp \Big\{ - s  \tfrac{k}{n}\Big( 1- \tsum_{j \geq 2} q^{j-1}\alpha_j   \Big)  - O( n^{-1/3})\tfrac{sk}{n} - O( s\tfrac{k^2 }{ n^2})  \Big\}. \nonumber
	\end{align}
Hence, {since $sk/n = O(1)$ in the regime we are concerned with,}
$$
 q \E[q^{a_{i+1} - a_i}| \cH_{i}] \ge \psi(1-q,c) (1 + o(1)) \ge \psi(1-q, c) + o(1)
$$
where the $o(1)$ term is non random and independent of $i$, and for the last inequality we have used that
	\begin{equation}\label{DCT}
	\exp ( - c \tsum_j q^{j-1} \alpha_j) \to \exp ( - c \tsum_j q^{j-1} \bar \alpha_j)
	\end{equation}
 which follows from the fact that $q \le q_0 <1$ and the dominated convergence theorem, as $ jk_j/k$ is uniformly bounded by $1$. Note that the above estimate is uniform in $i \geq 1$.

	For the upper bound, we use \eqref{eq:new_active_lb}. Let $\epsilon_n \to 0$ sufficiently slowly that $\eps_n n^{1/3}\to \infty$. For concreteness take $\eps_n = n^{-1/6}$.
	Define
	$G:=\{t\in\{1,\dots,s\}: \sum_{m\geq 2} m y^{(t)}_m \le \epsilon_n k\},$ and let $I = G^c$.  Packets $t \in I$ are the bad packets for which a significant fraction of the mass corresponding to that packet (at least $\eps_n$) was already discovered at step $i$; by contrast packets $t \in G$ are those for which a fraction at least $(1- \eps_n)$ remains to be discovered in the exploration. In the case where the conjugacy class contains only one type of cycles, say $k$-cycles, then $I$ coincides with the set of hyperedges already revealed. At the other end of the spectrum, when the conjugacy class $\Gamma$ is broken down into many small cycles, then $I$ is likely to be empty. But in all cases, $|I|$ satisfies the trivial bound
	$
	|I| \leq \tfrac{2n^{2/3}}{\eps_n k}
	$
by {definition of $T^\uparrow$}, and in particular
\begin{equation}\label{bad}
\tfrac{k |I| }{n } \le \tfrac{2}{\eps_n n^{1/3} } \le 2n^{-1/6} \to 0.
\end{equation}
 This turns out to be enough for our purposes.

 Note that $\E( q^{ \sum_{t=1}^s M_t (V_t -1 )} ) $ and $\E( q^{ n^{1/4} \wedge \sum_{t=1}^s M_t (V_t -1 )} ) $ can only differ by at most $q^{n^{1/4}}$, which is exponentially small in $n^{1/4}$ for a fixed $q\le q_0 <1$, so we can neglect this difference.
 Then we may write, counting only hyper edges from good packets, using the fact that $1- x \le e^{-x}$ for all $x\in \mathbb{R}$, and \eqref{bad} (recalling that $I_i$ is the indicator of the event that no self-intersection occurs among the first $n^{1/4}$ vertices connected to $w$):
	\begin{align}
		q \E[q^{a_{i+1} - a_i} |  \cH_i]& \leq 1- \E(I_i) + \E(I_i) \Big(q^{n^{1/4}} + \tprod_{t=1}^s\Big[1-\tfrac{k-\sum_{m \geq 2}my^{(t)}_m}{n}\Big(1- \tsum_{j \geq 2} q^{j-1} \tfrac{j(k_j-y^{(t)}_j)}{k-\sum_{m \geq 2}my^{(t)}_m}\Big)\Big] \Big) \nonumber \\
		&\leq  2n^{-1/2} + q^{n^{1/4}} + \tprod_{t\in G}\Big[1-\tfrac{k}{n}(1- \eps_n) \Big(1- \tsum_{j \geq 2} q^{j-1} \tfrac{jk_j}{k(1- \eps_n)}\Big) \Big] \nonumber \\
		& \leq o(1) + \exp\Big\{-s\tfrac{k}{n}(1- \eps_n)   + \tfrac{k}{n} |I| (1- \eps_n)  + s \tfrac{k}{n} \tsum_{j \geq 2} q^{j-1} \alpha_j \Big\} \nonumber \\
		&= o(1) +  \exp\Big\{-s\tfrac{k}{n} + \tfrac{sk}{n}\tsum_{j \geq 2} q^{j-1} \alpha_j \Big\}(1+ o(1) ) \label{control}\\
		& \leq \psi(1-q,c)+o(1)\nonumber
	\end{align}
 where the $o(1)$ term again is non random and uniform in $i\ge 1$, but might depend on $q$ (the last inequality again from comes from \eqref{DCT}).
 The proof is complete.
	\end{proof}

Lemma \ref{lemma:A_mgf} above tells us that, at the level of generating functions, the distribution of $a_{i+1} - a_i$ behaves very much like a sequence of i.i.d. random variables with distribution determined by $\psi$, even if we don't ignore self-intersections. It is thus easy to build martingales from quantities of the form $q^{a_i}$, which behave as if the increments of $a_i$ were i.i.d., at least until we reach size $n^{2/3}$. Hence this will allow us to reach a size of $n^{2/3}$ for $a_i$ almost as if there were no self-intersections, and so with probability approximately $\theta(c)$. Fundamentally, this is because even if self-intersections do occur, they are rare and do not cause a significant loss of mass. Technically, it is easier to have a separate argument for bringing the cluster to a polylogarithmic size before using this information to show that the cluster reaches size $n^{2/3}$ with essentially the same probability. This is what we achieve in Step 1, which we are now ready for.

\medskip \noindent \textbf{Step 1.} We show that the cluster containing a given vertex $v$ is at least logarithmically large with probability approximately $\theta(c)$, and furthermore the number of vertices for which this occurs is approximately $n \theta(c)$ in the sense of convergence in probability.
\begin{lemma}\label{L:theta}
Let $\cC_v$ denote the component containing $v$. We have that
\begin{equation}\label{Egiant}
\lim_{n \to \infty}\P( |\cC_v| >  (\log n)^2) = \theta(c).
\end{equation}
\end{lemma}

\begin{proof}

We start with the upper bound of \eqref{Egiant}, for which we simply make a comparison with a Galton--Watson process: to reach size $\log n$ the exploration process must survive more than a finite number of steps. More precisely, we make the following observation. Let $m \ge 1$ be some arbitrary fixed large integer, and observe $\P( |\cC_v| >  (\log n)^2) \le \P( |\cC_v|  \ge m)$ trivially. Now, whether $|\cC_v|$ reaches size $m$ is something that can be decided by performing the breadth-first search exploration of the cluster on a finite (at most $m$) number of steps: i.e., if we let $X_{i+1} = | A_{i+1} \setminus A_i | $, 
then a direct and crude consequence of Lemma \ref{lemma:A_mgf} is that $(X_1, \ldots, X_m)$ converge to i.i.d. random variables $(\bar X_1, \ldots, \bar X_m)$ (which are possibly improper, if $\sum \bar \alpha_j < 1$)  having as generating function $\E(q^{\bar X}) = \psi(1- q, c) $. Formally, the $\bar X_i$ have the same distribution as
$$
\bar X  = \Big( \tsum_j (j-1) \text{ Poisson } ( c \bar \alpha_j) \Big)+ \infty \cdot \text{ Poisson } ( c(1- \tsum_j \bar \alpha_j))
 $$
 where the Poisson random variables are independent. Consequently, if $W$ is the total progeny of a Galton--Watson branching process with offspring distribution $\bar X_i$ (note in particular that $W = \infty$ as soon as one nodes in the tree has offspring $\bar X_i = \infty$). We conclude that $\P( |\cC_v| \ge m) \to \P( W\ge m)$, and hence, taking the limsup and letting $m \to \infty$,
  $$
 \limsup_{n \to \infty} \P( T^\downarrow \ge (\log n)^2) \le \P( W  = \infty) = \theta(c).
 $$
This proves the upper bound in \eqref{Egiant}.

We now discuss the lower bound to \eqref{Egiant}, which is essentially the same argument, together with the observation that self-intersections are unlikely to occur before $(\log n)^2$ vertices have been explored. For this we can assume without loss of generality that $\theta(c)>0$, otherwise there is nothing to prove. Let
$$
T_1 = \inf\{i\ge 1: a_i > (\log n)^2\};
$$
we will prove the slightly stronger result that $\liminf_{n \to \infty} \P( T_1 < T_\downarrow) \ge \theta(c)$. (This is slightly stronger, because $|\cC_v|$ could in principle be greater than $(\log n)^2$ without the active set ever reaching that size).
Let $X_i$ be i.i.d. random variables with generating function given by
\begin{equation}\label{offspring2}
\psi_n(q) =  \E( q^{X_i}) = \tprod_{t=1}^s ( 1 - \tfrac{k}{n} (1 - \tsum_j q^{j-1} \alpha_j ) ),
\end{equation}
so that, by \eqref{Vt}, $a_{1} - a_0 $ has the same distribution as $X_1$ when $A_0  = \{v\}$ (see e.g. \eqref{offspring} where a similar calculation is carried). We can use the random variables $X_i$ to generate the breadth first exploration of $\cC_v$ until we find a self-intersection. Thus let $\tilde Y_i$ be a collection of randomly chosen vertices of $\{1, \ldots, n\}$ of size $X_i$, and at each time step, add to the active set $\tilde A_{i+1}$ the set $\tilde Y_i$ and remove the currently explored vertex. Then we can couple $A_i$ and $\tilde A_i$ so that $A_i = \tilde A_i$ until the first time $T_{\text{inter}}$ such that $\tilde Y_i \cap (\tilde Y_j \cup \{v\}) \neq \emptyset$ for some $i \neq j \le T_{\text{inter}}$. Furthermore, until $T_{\text{inter}}$, $\tilde A_i$ is the breadth-first search exploration of a branching process with offspring distribution  \eqref{offspring2}. It becomes extinct with a probability $q_n$, and we claim that $q_n$ satisfies $q_n \to 1- \theta(c)$ as $n \to \infty$ by \eqref{DCT}. Indeed, $\psi_n$ clearly converges uniformly to $\psi(\cdot, c)$ on $[0, x_0]$ for $x_0 <1$ by Lemma \ref{lemma:A_mgf} and this is the regime we are interested in since by assumption $\theta(c)>0$.

Hence, it is clear that if $W_n$ is the total progeny of this branching process, then $\P( W_n \ge (\log n)^2)  \ge \P( W_n = \infty) = 1- q_n \to \theta(c)$, and combining with the argument in the upper bound on \eqref{Egiant} we deduce that $\P( W_n \ge (\log n)^2) \to \theta(c)$. On the other hand, $T_1 < T_{\text{inter}}$ with probability tending to 1 as $n \to \infty$ by the birthday problem, and so in fact $\P( T_1 < T_\downarrow) = \P( W_n \ge (\log n)^2 ) + o(1)$, so we are done.
\end{proof}

It is important to note that self-intersections may occur at the very step that $a_i$ exceeds $(\log n)^2$ (for instance, think about the case when the conjugacy class has some of its mass coming from cycles larger than $n^{1/2}$: discovering such a cycle would immediately produce a self-intersection). Even so, the active set reaches size $(\log n)^2$ before such a self-intersection is discovered.

\medskip As announced at the beginning of Step 1, we complement this with a law of large numbers:
\begin{lemma}
\begin{equation}\label{lln_giant}
	\tfrac1n | \{ v : | \cC_v| \ge (\log n)^2 \} |\to \theta(c)
\end{equation}
in probability as $n \rightarrow \infty$.
\end{lemma}

\begin{proof}
Let $Z = \sum_{v = 1}^n 1_{\{| \cC_v | \ge (\log n)^2 \}}$, so by the previous lemma we know that $\E(Z) / n \to \theta$ by \eqref{Egiant}. Hence if we show that $\var (Z) \le \eps n^2$ for any $\eps>0$ and any $n$ sufficiently large, then \eqref{lln_giant} follows by Chebyshev's inequality. In particular, it suffices to show that for $v \neq w \in \{1, \ldots, n \}$,
$$
\limsup_{n \to \infty} \text{Cov}( 1_{\{| \cC_v | \ge (\log n)^2 \}}, 1_{\{| \cC_w | \ge (\log n)^2 \}} ) \le  0
$$
or equivalently,
\begin{equation}\label{cov}
\limsup_{n \to \infty} \P( | \cC_v | \ge (\log n)^2 , | \cC_w | \ge (\log n)^2 ) \le \theta(c)^2.
\end{equation}
On the other hand, \eqref{cov} can be proved in exactly the same way as the upper bound of \eqref{Egiant} above: for both $| \cC_v |$ and $|\cC_w|$ to be larger than $(\log n)^2 $, both must be greater than $m$ where $m \ge 1$ is fixed. This is an event which depends on a finite number of steps  (at most $2m$) in the explorations of $\cC_v$ and $\cC_w$, and so can be approximated by Lemma \ref{lemma:A_mgf} by the same event for two independent branching processes. Letting $m \to \infty$ finishes the proof.
\end{proof}

For the rest of the proof we now assume that $c> c_\Gamma$ so that $\theta(c) >0$. Hence fix $q \in [0,1)$ such that $\psi(1-q, c)/ q < 1$, and note that using Lemma \ref{lemma:A_mgf}, we can suppose that, for some fixed $\epsilon >0$, $n$ is large enough so that
\begin{equation}\label{x}
			\E[q^{a_{i+1} - a_i} | \cH_{i}] \leq (1+ \epsilon)^{-1}
\end{equation}
		almost surely on $\{T>i\}$.

\medskip\noindent
\textbf{Step 2.}		We now extrapolate the information obtained in the previous step to show that, still with probability approximately $\theta(c)$, the active set of $\cC_v$ can reach a size of at least $O(n^{2/3})$. To do so we suppose our exploration from Step 1 yields an active set of size at least $(\log n)^2$ (which, as discussed, occurs with probability $\theta(c) + o(1)$. We will restart the exploration from that point on, calling this time $i=0$ again. Hence the setup is the same as before, except that at time $i=0$ we have $a_0 = \lfloor (\log n)^2 \rfloor $: we only keep the first $(\log n)^2$ of the active vertices discovered at time $T_1$, and declare all further active vertices at time $T_1$ to be removed at time $i=0$ in the exploration of Step 2.

Recall our notations for $T^\downarrow$ and $T^\uparrow$ in \eqref{Tdown} and \eqref{Tup}. Our goal in this step is to show the following control:

\begin{lemma}
suppose that given $\cH_0$, it is a.s. the case that $a_0 = \lfloor (\log n)^2 \rfloor$, and $r_0 \le n^{2/3}$.Then
\begin{equation}\label{step2}
\P( T^\downarrow  < n^{2/3} \wedge T^\uparrow| \cH_0) = O(q^{(\log n)^2}) = o(n^{-1}).
\end{equation}
 \end{lemma}
\begin{proof}
Set $S =n^{2/3} \wedge T^\uparrow \wedge T^\downarrow$ and for $i \geq 0$, let
$$M_i:= q^{a_{i\wedge S}}(1+\epsilon)^{i\wedge S },$$
so $M=(M_i:i=0,\dots )$ is a supermartingale in the filtration $(\cH_0, \cH_{1}, \ldots)$. Observe that $S \le n^{2/3}$ so $M$ is bounded. Note that on the event $\{ S = T^\downarrow\}$,
$$
M_T= (1+ \eps)^{T^\downarrow } \ge 1_{\{ S = T^\downarrow\}}
$$
hence by the optional stopping theorem (since $M$ is bounded), given $\cH_0$ and under the assumptions of the lemma on $\cH_0$,
\begin{align*}
\P( S  = T^\downarrow | \cH_0) & \le \E( M_S  1_{\{ S = T^\downarrow\}} | \cH_0) \\
& \le M_0 =q^{a_0} \le q^{(\log n)^2 -1},
\end{align*}
as desired.
\end{proof}

Consequently, since the error bound in Lemma \ref{step2} is $o(n^{-1})$, we deduce that if
\begin{equation}\label{tildeG}
\cG = \{ v: |\cC_v(s)| > n^{2/3} \}, \tilde \cG = \{ v: |\cC_v(s) | > (\log n)^2 \},
\end{equation}
then $\cG = \tilde \cG$ with high probability, and hence in particular
\begin{equation} \label{eq:step2}
\frac{|\cG|}{n} \to \theta(c)
\end{equation}
in probability as $n \to \infty$.

\medskip \textbf{Step 3.}	
We now show that if $v$ and $v'$ are two vertices such that $\cC_v = \cC_v(s)$ and $\cC_{v'} = \cC_{v'}(s)$ are both larger at time $s$ than $n^{2/3}$ then they are highly likely to be connected at some slightly later time $s+s'$.
This follows from a so-called ``sprinkling'' argument, as follows.
That is, suppose we add $s'$ packets, with
$$
s' = \left\lceil \frac{ D n^{2/3} \log n}{k} \right\rceil
$$
for some $D>0$ to be chosen later on.
 Note that $s'k/n \to 0$ so that $ (s+ s') k /n \to c$. Since $s = s(n)$ is an arbitrary sequence such that $s k / n \to c$ it suffices to show that $v$ and $v'$ are then connected at time $s+ s'$. In fact we will check that the two clusters can be connected using smaller edges that the hyperedges making each packet, as follows. For each hyperedge of size $j$ we will only reveal a subset of $\lfloor j/2 \rfloor $ edges (of size 2) with disjoint support.
 Since $\lfloor j/2 \rfloor \ge j /3$ for any $j \ge 2$, this gives us at least $k/3$ edges
 for each packet; these are sampled uniformly at random without replacement from $\{1, \ldots, n\}$. We will check that a connection occurs between the two clusters within these $s' k /3 $ edges, with high probability.

Call the two clusters $A$ and $A'$ for simplicity; these are two arbitrary sets of size $\lfloor n^{2/3}\rfloor $ which we can assume to be disjoint otherwise there is nothing to prove.  Call a packet of edges good if their intersections with each of $A$ and $A'$ contains at most $\lfloor n^{2/3} /2 \rfloor $ vertices, and call it bad otherwise. We reveal the edges in a given packet one by one, sampling without replacement. Note that so long as packet of edges has not been observed to be bad, the probability that the next edge connects $A$ and $A'$ is at least $n^{4/3}/(16(n-k)^2) \ge n^{-2/3}/32$. (Note that if $k \le n^{2/3} /2$ then every packet is necessarily good). Hence the probability that no connection between $A$ and $A'$ occurs for a good packet is at most
$$
(1- n^{-2/3}/32)^{k/3} \le \exp ( - \frac{kn^{-2/3}}{96}).
$$
Now, each packet is bad independently of each other, with probability tending to 0 by Markov's inequality (since the expected intersection of a pack of edges with $A$ is at most $|A| k /n = o (|A|)$) and hence less than $1/2$ say. So by standard Chernoff bounds on Binomial random variables, with probability at least $1-\exp( - h \times s')$ (where $h>0 = (1/4) \log (1/4) + (3/4) \log (3/4) + \log 2 $ is a universal constant), at least $s'/4$ packs are good. Putting together these two observations, we deduce that the probability that there are no connections between $A$ and $A'$ after $s'$ packs of edges have been added is at most
\begin{equation}\label{eq:step3}
\exp ( - \frac{kn^{-2/3}}{96} \frac{s'}{4}) + \exp ( - h s') \le \exp ( - \frac{D}{400} \log n) + \exp ( - h s') ,
\end{equation}
By choosing $D = 1201$, this is $o(n^{-3})$ at least if $k \ge n^{2/3}/2$ (so that $s' \ge 2 D \log n$). However, if $k \le n^{2/3}$, then every packet is good, and so \eqref{eq:step3} holds without the second term on the right hand side. Either way,
\begin{equation}\label{eq:step3concl}
\P( \cC_v(s+s') \cap \cC_{v'}(s+ s') = \emptyset ) = o(n^{-3}).
\end{equation}

\begin{proof}[Proof of Theorem \ref{lemma:hyper_k_cycle}]
 We are now ready to conclude that vertices are either in small component at time $s$ or connected at time $s+ s'$. Recall our notation $\mathcal{G} = \{ v : |\cC_v(s)| > n^{2/3}\}$. Then by \eqref{eq:step2}, we know that $|\cG|/n\to \theta(c)$ in probability as $n \to \infty$. We now aim to show that $\cG$ is connected at time $s+s'$, with high probability. For $v, v' \in \{1,\dots,n\}$, write $v \leftrightarrow v'$ to indicate that $v$ is connected to $v'$. Then by Step 3 (more specifically, \eqref{eq:step3concl}),
 \[
 \P \left( \left. \bigcup_{v,v' \in \cG} \{v \leftrightarrow v' \text{  at time $s+s'$}\}^c \right| H_s \right) \le |\cG|^2 o(n^{-3}) \le o(n^{-1}).
		\]
Hence $\cG$ is entirely connected at time $s+s'$ with probability tending to 1. This proves that $H_{s+s'}$ contains a component of relative size  converging to $\theta(c)$ in probability. Let us now check that every other component at time $s+s'$ is small. Note that since $\cG = \tilde \cG$ with probability tending to one (where $\tilde \cG$ is defined in \eqref{tildeG}), any component disjoint from $\cG$ at time $s+s'$ must have been smaller than $(\log n)^2$ at time $s$. Since at most $s'k$ connections are added, this means that, on the event $\cG = \tilde \cG$, the maximal size of a component at time $s+s'$ disjoint from $\cG$ is smaller than $s' k (\log n)^2 \le D n^{2/3} (\log n)^3$. This shows that every other component is $O(n^{2/3} ( \log n)^3)$ on an event of high probability.

The proof of Theorem
\ref{lemma:hyper_k_cycle} is complete, since $s+s'$ in an arbitrary sequence such that $(s+s')k/n \to c$.
\end{proof}

  \subsection{Poisson--Dirichlet structure}
  The renormalised cycle lengths $\X(\sigma)$ of a permutation $\sigma\in \Sn$ is the cycle lengths of $\sigma$ divided by $n$, written in decreasing order. In particular we have that $\X(\sigma)$ takes values in
  \begin{equation}\label{Omega}
    \Omega_\infty:=\{\mathbf (x_1 \geq x_2 \geq \dots ): x_i \in [0,1] \text{ for each }i \geq 1 \text{ and } \sum_{i=1}^\infty x_i =1  \}.
  \end{equation}
  We equip $\Omega_\infty$ with the topology of pointwise convergence. If $\sigma_n$ is uniformly distributed in $\Sn$ then $\X(\sigma_n) \rightarrow Z$ in distribution as $n \rightarrow \infty$ where $Z$ is known as a Poisson--Dirichlet random variable. It can be constructed as follows. Let $U_1,U_2,\dots$ be i.i.d. uniform random variables on $[0,1]$. Let $Z^*_1=U_1$ and inductively for $i \geq 2$ set $Z^*_i=U_i(1-\sum_{j=1}^{i-1} Z^*_j)$. Then $(Z^*_1,Z^*_2,\dots)$ can be ordered in decreasing size and the random variable $Z$ has the same law as $(Z^*_1,Z^*_2,\dots)$ ordered by decreasing size.

  The next result is a generalisation of Theorem 1.1 in \cite{schramm_transpo} to the case of general conjugacy classes. The proof is a simple adaptation of the proof of Schramm and we provide  the details in an appendix.

  \begin{thm}\label{thm:general_schramm}
    Suppose $s=s(n)$ is such that $sk/n \rightarrow c$ as $n \rightarrow \infty$ for some $c>c_\Gamma$. Then we have that for any $m\in \N$
    \[
      \left(\frac{\X_1(X_s)}{\theta(c)}, \dots, \frac{\X_m(X_s)}{\theta(c)}\right) \rightarrow (Z_1,\dots,Z_m)
    \]
    in distribution as $n \rightarrow \infty$ where $Z=(Z_1,Z_2,\dots)$ is a Poisson--Dirichlet random variable.
  \end{thm}

  \section{Proof of curvature theorem}

\label{sec:curv}
  \subsection{Proof of the upper bound on curvature}\label{sec:upper_bound}

  We claim that it is enough to show the upper bound for $c>c_\Gamma$ in \eqref{eq:thm_lower2}. Indeed, notice that $c \mapsto \kappa_c$ is nondecreasing. Hence let $c\leq c_\Gamma$ and suppose we know that $\limsup_{n\rightarrow\infty}\kappa_{c'} \leq \theta(c')^2$ holds for all $c' > c_\Gamma$. Then we have that $\limsup_{n \rightarrow \infty}\kappa_c \leq \theta(c')^2$ for each $c'> c_\Gamma$. Taking $c' \downarrow c_\Gamma$ and using the fact that $\lim_{c' \downarrow c_\Gamma} \theta(c')=0$ shows that $\lim_{n\rightarrow\infty}\kappa_c=0$.

  Fix $c>c_\Gamma$ and let $t:=\lfloor cn/k\rfloor$. We need to show the upper bound in \eqref{eq:thm_lower2}. In other words, we wish to prove that for some $\sigma, \sigma' \in \Sn$
  $$
  \liminf_{n \rightarrow \infty} \frac{W_1(X^\sigma_{t}, X^{\sigma'}_{t})}{d(\sigma, \sigma')} \geq  1- \theta(c)^2.
  $$
  We will choose $\sigma = \id$ and $\sigma' =  \tau_1 \circ \tau_2$, where $\tau_1, \tau_2$ are independent uniformly chosen transpositions. To prove the lower bound on the Kantorovitch distance we use the dual representation of the distance $W_1(X,Y)$ between two random variables $X,Y$:
  \begin{equation}\label{eq:wasserstein_dual}
    W_1(X,Y)= \sup\{\E[f(X)]-\E[f(Y)]: f \text{ is Lipschitz with Lipschitz constant }1\}.
  \end{equation}
  Let $f(\sigma) = d(\id, \sigma)$ be the distance to the identity (using only transpositions, as usual). Then observe that $f$ is 1-Lipschitz.
  It suffices to show
  \begin{equation}\label{eq:kappa_expectation_bnd0}
    \liminf_{n\to \infty} \E[f(X^{\tau_1\circ \tau_2}_{t})]-\E[f(X^{\id}_{t})] \geq 2 (1- \theta(c)^2).
  \end{equation}

  We will now show \eqref{eq:kappa_expectation_bnd0} by a coupling argument.
  Construct the two walks $X^{\tau_1\circ\tau_2}$ and $X^{\id}$ as follows. Let $\gamma_1,\gamma_2,\dots$ be a sequence of i.i.d. random variables uniformly distributed on $\Gamma$, independent of $(\tau_1, \tau_2)$.
  Using Lemma \ref{lemma:transpo_end} with $\sigma_0 = \tau_1 \circ \tau_2$, which is independent of $X^{\id}$, we can construct $X^{\tau_1\circ \tau_2}_{t}$ as
  \[
    X^{\tau_1\circ \tau_2}_{t}=\gamma_1\circ\dots\circ\gamma_t \circ \tau_1\circ\tau_2.
  \]
  Next we couple $X^{\id}_{t}$ by constructing it as
  \[
    X^{\id}_{t}= \gamma_1\circ \dots \circ \gamma_{t}.
  \]
  Thus under this coupling we have that $X^{\tau_1\circ \tau_2}_{t}=X^{\id}_{t} \circ \tau_1\circ\tau_2$. Let $X=X^{\id}$, then from \eqref{eq:kappa_expectation_bnd0} the problem reduces to showing
  \begin{equation}\label{eq:kappa_expectation_bnd}
    \liminf_{n \rightarrow \infty}\E[d(\id,X_{t}\circ\tau_1\circ\tau_2)-d(\id,X_{t})]\geq  2 (1- \theta(c)^2).
  \end{equation}

We recall that a transposition can either induce a fragmentation or a coalescence of the cycles. Indeed, a transposition involving elements from the same cycle generates a fragmentation of that cycle, and one involving  elements from different cycles results in the cycles being merged. (This property is the basic tool used in the probabilistic analysis of random transpositions, see e.g. \cite{BerestyckiDurrett} or \cite{schramm_transpo}).  Hence either $\tau_{1}$ fragments a cycle of $X_{t}$ or $\tau_{1}$ coagulates two cycles of $X_{t}$. In the first case, $d(\id,X_t \circ \tau_1) = d(\id, X_t \circ \tau_1) - 1$, and in the second case we have $d(\id,X_t\circ\tau_1) = d(\id,X_t\circ \tau_1) + 1$.
  Let $F$ denote the event that $\tau_{1}$ causes a fragmentation. Then
  \begin{align*}
    \E[d(\id,X_t\circ\tau_1)-d(\id,X_{t})] &= 1 - 2 \P(F).
  \end{align*}
  Using the Poisson--Dirichlet structure described in Theorem \ref{thm:general_schramm} it is not hard to show that $\P(F) \to \theta(c)^2/2$ (see, e.g., Lemma 8 in \cite{berestycki_k}). Applying the same reasoning to $X_{t} \circ \tau_1\circ\tau_2$ and $X_{t}\circ \tau_1$ we deduce that
  $$
  \lim_{n \to \infty} \E[d(\id,X_t\circ\tau_1\circ\tau_2)-d(\id,X_{t})] =2(1- \theta(c)^2)
  $$
  from which the lower bound \eqref{eq:kappa_expectation_bnd} and in turn the upper bound in \eqref{eq:thm_upper} follow readily.

  \subsection{Proof of lower bound on curvature.}\label{sec:lower_bound}

  We now assume that $c>c_\Gamma$ and turn out attention to the lower bound on the Ricci curvature, which is the heart of the proof. Throughout we let $k=|\Gamma|$ and $t=\lfloor cn/k \rfloor$. With this notation in mind we wish to prove that
  $$
  \limsup_{n \to \infty}\,\sup_{\sigma,\sigma'} \frac{\E d(X^{\sigma}_{t }, X^{\sigma'}_{t } )}{d(\sigma, \sigma')} \le \alpha: =  1-\theta(c)^4
  $$
  for some appropriate coupling of $X^\sigma$ and $X^{\sigma'}$, where the supremum is taken over all $\sigma,\sigma'$ with even distance. Note that we can make several reductions: first, by vertex transitivity we can assume $\sigma = \id$ is the identity permutation. Also, by the triangle inequality (since $W_1$ is a distance), we can assume that $\sigma'= (i,j )\circ (\ell,m)$ is the product of two distinct transpositions. There are two cases to consider:  either the supports of the transpositions are disjoint, or they overlap on one vertex. We will focus in this proof on the first case where the support of the transpositions are disjoint; that is, $i,j,l,m$ are pairwise distinct. The other case is dealt with very much in the same way (and is in fact a bit easier).

  Clearly by symmetry $\E d(X^{\id}_{t}, X^{(i,j)\circ (\ell,m)}_{t} )$ is independent of $i$, $j$, $\ell$ and $m$, so long as they are pairwise distinct. Hence it is also equal to $\E d(X^{\id}_{t}, X^{\tau_1\circ\tau_2}_{t} )$ conditioned on the event $A$ that $\tau_1,\tau_2$ having disjoint support, where $\tau_1$ and $\tau_2$ are independent uniform random transpositions.
  This event has an overwhelming probability for large $n$, thus it suffices to construct a coupling between $X^{\id}$ and $X^{\tau_1\circ\tau_2}$ such that
  \begin{equation}
    \label{goal0}
    \limsup_{n \to \infty} \E d(X^{\id}_{t }, X^{\tau_1\circ\tau_2}_{t } ) \le 2(1- \theta(c)^4).
  \end{equation}
  Indeed, it then immediately follows from stochastic commutativity (Lemma \ref{lemma:transpo_end}) that the same is true with the expectation replaced by the conditional expectation given $A$, since the distance is bounded by two.

  Next, let $X$ be a random walk on $\Sn$ which is the composition of i.i.d. uniform elements of the conjugacy class $\Gamma$.
  We decompose the random walk $X$ into a walk $\tilde X$ which evolves by applying transpositions at each step as follows. For $t =0,1, \ldots, $ write out
  \[
    X_t= \gamma_1\circ \dots\circ \gamma_t
  \]
  where $\gamma_1,\gamma_2,\dots$ are i.i.d. uniformly distributed in $\Gamma$. As before we decompose each step $\gamma_s$ of the walk into a product of cyclic permutations, say
  \begin{equation}\label{decomp_cyc}
    \gamma_s = \gamma_{s,1} \circ \ldots \circ \gamma_{s,r}
  \end{equation}
  where $r = \sum_{j\ge 2} k_j$. The order of this decomposition is irrelevant and can be chosen arbitrarily. For concreteness, we decide that we start from the cycles of smaller sizes and progressively increase to cycles of larger sizes. We will further decompose each of these cyclic permutation into a product of transpositions, as follows: for a cycle $c = (x_1, \ldots, x_j)$, write
  $$
  c= (x_1, x_2) \circ \ldots  \circ (x_{j-1}, x_j).
  $$
  This allows to break any step $\gamma_s$ of the random walk $X$ into a number $$\rho := \sum_j (j-1) k_j$$ of elementary transpositions, and hence we can write
  \begin{equation}\label{eq:gamma_s}
    \gamma_s=\tau^{(1)}_s\circ\dots\circ\tau^{(\rho)}_s
  \end{equation}
  where $\tau^{(j)}_s$ are transpositions.  Note that the vectors $( \tau^{(i)}_s; 1\le i \le \rho)$ in \eqref{eq:gamma_s} are independent and identically distributed for $s = 1,2, \ldots$ and for a fixed $s$ and $1\le i \le \rho$, $\tau^{(i)}_s$ is a uniform transposition, by symmetry. However it is important to observe that the transpositions $\tau_s^{(i)}$ are \emph{not} independent. Nevertheless, they obey a crucial conditional uniformity which we explain now. First we have to differentiate between the set of times when a new cycle starts and the set of times when we are continuing an old cycle.

  \begin{definition}[Refreshment Times] \label{def:refreshment}
    We call a time $s$ a refreshment time if $s$ is of the form $s=\rho \ell + \sum_{j=2}^m (j-1) k_j$ for some $\ell \in \N\cup \{0\}$ and $m \in\N \backslash \{1\}$.
  \end{definition}
  We see that $s$ is a refreshment time if the transposition being applied to $\tilde X$ at time $s$ is the start of a new cycle. Using this we can describe the law of the transpositions being applied to $\tilde X$.

  \begin{prop}[Conditional Uniformity]\label{prop:conditional_unif}
    For $s \in \N$ and $i \leq \rho$, the conditional distribution of $\tau^{(i)}_s$ given $\tau^{(1)}_s, \ldots, \tau^{(i-1)}_s$ can be described as follows. We write $\tau^{(i)}_s  = (x,y)$ and we will distinguish between the first marker $x$ and the second marker $y$. There are two cases to consider:
    \begin{enum}
    \item $s\rho + i$ is a refreshment time and thus $\tau^{(i)}_s$ corresponds to the start of a new cycle
    \item $s\rho + i$ is not a refreshment time and so $\tau^{(i)}_s$ is the continuation of a cycle.
    \end{enum}
     In case (i) $x$ is uniformly distributed on $S_i: = \{1, \ldots, n\} \setminus \text{Supp} (\tau_s^{(1)} \circ \ldots \circ \tau^{(i-1)}_s)$ and $y$ is uniformly distributed on $S_i \setminus\{x\}$. In case (ii) $x$ is equal to the second marker of $\tau^{(i-1)}_s$ and $y$ is uniformly distributed in $S_i$.
  \end{prop}

  Note that in either case, the second marker $y$ is conditionally uniformly distributed among the vertices which have not been used so far. This conditional independence property is completely crucial, and allows us to make use of methods (such as that of Schramm \cite{schramm_transpo}) developed initially for random transpositions) for general conjugacy classes, so long as $|\Gamma| = o(n)$. Indeed in that case the second marker $y$ itself is not very different from a uniform random variable on $\{1, \ldots, n\}$.

  We will study this random walk using this new \emph{transposition time scale}. We thus define a process $\tilde X=(\tilde X_u: u = 0,1,\ldots)$ as follows. Let $u\in \{ 0,1, \ldots\} $ and write $u=s\rho +i$ where $s,i$ are nonnegative integers and $i<\rho$. Then define
  \begin{equation}\label{eq:tilde_X_defn}
    \tilde X_u:= X_s \circ \tau_{s+1}^{(1)} \circ \dots \circ \tau_{s+1}^{(i)}.
  \end{equation}
  Thus it follows that for any $s \geq 0$, $\tilde X_{s\rho }=X_s$.
  Notice that $\tilde X$ evolves by applying successively transpositions with the above mentioned conditional uniformity rules.

  Now consider our two random walks, $X^{\id}$ and $X^{\tau_1\circ\tau_2}$ respectively, started respectively from $id$ and $\tau_1 \circ \tau_2$, and let $\tilde X^{\id}$ and $\tilde X^{\tau_1 \circ \tau_2}$ be the associated processes constructed using \eqref{eq:tilde_X_defn}, on the transposition time scale. Thus to prove \eqref{goal0} it suffices to construct an appropriate coupling between $\tilde X^{\id}_{t \rho}$ and $\tilde X^{\tau_1 \circ \tau_2}_{t\rho}$.
  Next, recall that for a permutation $\sigma \in \Sn$, $\X(\sigma)$ denotes the renormalised cycle lengths of $\sigma$, taking values in $\Omega_\infty$ defined in \eqref{Omega}. The walks $\tilde X^{\id}$ and $\tilde X^{\tau_1\circ\tau_2}$ are invariant by conjugacy and hence both are distributed uniformly on their conjugacy class. Thus ultimately it will suffice to couple $\X(\tilde X^{\id}_{t \rho })$ and $\X(\tilde X^{\tau_1\circ\tau_2}_{t \rho })$.

  Fix $\delta>0$ and let $\Delta=\lceil \delta^{-9} \rceil$. Define
  \begin{align*}
    s_1&= \lfloor (cn-\delta^{-9})/k\rfloor \rho\\
    s_2&=s_1+\Delta\\
    s_3&=t\rho.
  \end{align*}
  Our coupling consists of three intervals $[0,s_1]$, $(s_1,s_2]$ and $(s_2,s_3]$.

  Let us informally describe the coupling before we give the details. In what follows we will couple the random walks $\tilde X^{\id}$ and $\tilde X^{\tau_1\circ \tau_2}$ such that they keep their distance constant during the time intervals $[0,s_1]$ and $(s_2,s_3]$. In particular we will see that at time $s_1$, the walks $\tilde X^{\id}$ and $\tilde X^{\tau_1\circ \tau_2}$ will differ by two independently uniformly chosen transpositions. Thus at time $s_1$ most of the cycles of $\tilde X^{\id}$ and $\tilde X^{\tau_1\circ \tau_2}$ are identical but some cycles may be different. We will show that given that the cycles that differ at time $s_1$ are all reasonably large, then we can reduce the distance between the two walks to zero during the time interval $(s_1,s_2]$. Otherwise if one of the differing cycles is not reasonably large, then we couple the two walks to keep their distance constant during the time interval $[0,s_1]$, $(s_1,s_2]$ and $(s_2,s_3]$.

  More generally, our coupling has the property that $d(X^{\id}_t, X^{\tau_1 \circ \tau_2}_t)$ is uniformly bounded, so that it will suffice to concentrate on events of high probability in order to get a bound on the $L^1$-Kantorovitch distance $W(X^{\id}_t, X^{\tau_1 \circ \tau_2}_t )$.

  \subsubsection{Coupling for \texorpdfstring{$[0,s_1]$}{[0,s1]}}

  First we describe the coupling during the time interval $[0,s_1]$. Let $\tilde X=(\tilde X_s:s \geq 0)$ be a walk with the same distribution as $\tilde X^{\id}$, independent of the two uniform transpositions $\tau_1$ and $\tau_2$. Then we have that by Lemma \ref{lemma:transpo_end} for any $s \geq 0$, $\tilde X^{\tau_1\circ\tau_2}_s$ has the same distribution as $\tilde X_s \circ \tau_1 \circ \tau_2$. Thus we can couple $\X(\tilde X^{\id}_{s_1 })$ and $\X(\tilde X^{\tau_1\circ \tau_2}_{s_1 })$ such that
  \begin{align}\label{A1}
    \X(\tilde X^{\id}_{s_1 })&= \X(\tilde X_{s_1 }) \nonumber\\
    \X(\tilde X^{\tau_1\circ\tau_2}_{s_1 }) &= \X(\tilde X_{s_1 } \circ \tau_{1} \circ \tau_{2}).
  \end{align}

\subsubsection{Coupling for \texorpdfstring{$(s_1,s_2]$}{(s1,s2]}}\label{subsec:I_2}

For $s \geq 0$ define $\bar X_s=\X(\tilde X^{\id}_{s+s_1 })$ and $\bar Y_s = \X(\tilde X^{\tau_1\circ\tau_2}_{s+s_1 })$. Here we will couple $\bar X_s$ and $\bar Y_s$ for $s =0,\dots, \Delta$. During this time we aim to show that the discrepancies between $\bar X_0$ and $\bar Y_0$ resulting from performing the transpositions $\tau_1$ and $\tau_2$ at the end of the previous phase can be resolved. Our main tool for doing this will be a variant of a coupling of Schramm \cite{schramm_transpo}, which was already used in \cite{berestycki_k}.

At each step $s$ we try to create a matching between $\bar X_s$ and $\bar Y_s$ by matching an element of $\bar X_s$ to at most one element of $\bar Y_s$ of the same size. At any time $s$ there may be several entries that cannot be matched. By parity the combined number of unmatched entries is an even number, and observe that this number cannot be equal to two. Now $\tilde X^{\id}_{s_1}$ and $\tilde X^{\tau_1\circ\tau_2}_{s_1}$ differ by two transpositions as can be seen from \eqref{A1}. This implies that in particular initially (i.e., at the beginning of $(s_1,s_2]$), there are four, six or zero unmatched entries between $\bar X_0$ and $\bar Y_0$.

Fix $\delta>0$ and let $A(\delta)$ denote the event that the smallest unmatched entry between $\bar X_{0}$ and $\bar Y_0$ has size greater than $\delta>0$. We will show that on the event $A(\delta)$ we can couple the walks such that $\bar X_{\Delta}=\bar Y_{\Delta}$ with high probability. On the complementary event $A(\delta)^c$, couple the walks so that their distance remains $O(1)$ during the time interval $(s_1,s_2]$, similar to the coupling during $[0,s_1]$.

It remains to define the coupling during the time interval $(s_1,s_2]$ on the event $A(\delta)$. We begin by estimating the probability of $A(\delta)$.

\begin{lemma}
  \label{L:unmatched_large_0}  For any $c>1$ and $\delta>0$,
  $$
  \liminf_{n \to \infty} \P(A(\delta)) \ge [\theta(c) (1-p(\delta))]^4 .
  $$
  wher $p(\delta) \to 0$ as $\delta \to 0$.
\end{lemma}
\begin{proof}
  Recall that by construction $\bar X_0$ and $\bar Y_0$ only differ because of the two transpositions $\tau_{1} $ and $\tau_{2}$ appearing in \eqref{A1}.

  Recall the hypergraph $H_{s_1/\rho}$ on $\{1,\dots,n\}$ defined in the beginning of Section \ref{SS:hyper}. Since $c>c_\Gamma$, by Theorem \ref{lemma:hyper_k_cycle}, $H_{s_1/\rho}$ has a (unique) giant component with high probability. Let $A_1$ be the event that the four points composing the transpositions $\tau_{1}, \tau_{2}$ fall within the largest component of the associated hypergraph $H_{s_1/\rho}$. Since the relative size of the giant component converges in probability $\theta(c)$ by Lemma \ref{lemma:hyper_k_cycle}, note that $\P(A_1) \rightarrow \theta(c)^4$.

  Furthermore, it follows from Theorem \ref{thm:general_schramm} that conditionally on the event $A_1$, the asymptotic relative size of the cycles containing the four points making the transpositions $\tau_1, \tau_2$ can be thought of as the size of four independent samples from a Poisson-Dirichlet distribution, multiplied by $\theta(c)$. Hence the lemma is proved with $p(\delta)$ being the probability that one of the four samples has a size smaller than $\delta /\theta(c)$. Clearly $p(\delta ) \to 0$ so the result is proved.
\end{proof}

Recall that the transpositions which make up the walks $\tilde X^{\id}$ and $\tilde X^{\tau_1\circ\tau_2}$ obey what we called conditional uniformity in Proposition \ref{prop:conditional_unif}. For the duration of $(s_1,s_2]$ we will assume the \textit{relaxed conditional uniformity} assumption, which we describe now.

\begin{definition}[Relaxed Conditional Uniformity]\label{defn:relaxed}
  For $s=s_1+1,\dots,s_2$ suppose we apply the transposition $(x,y)$ at time $s$. Then
  \begin{enum}
  \item if $s$ is a refreshment time then $x$ is chosen uniformly in $\{1,\dots,n\}$,
  \item if $s$ is not a refreshment time then $x$ is taken to be the second marker of the transposition applied at time $s-1$.
  \end{enum}
  In both cases we take $y$ to be uniformly distributed on $\{1,\dots,n\}\backslash\{x\}$.
\end{definition}

In making the relaxed conditional uniformity assumption we are disregarding the constraints on $(x,y)$ given in Proposition \ref{prop:conditional_unif}. However the probability we violate this constraint at any point during the interval $(s_1,s_2]$ is at most $2(s_2-s_1)\rho/n=2\Delta k/n \to 0 $; and on the event that this constraint is violated the distance between the random walks can increase by at most $(s_2-s_1)=\Delta$. Hence we can without a loss of generality assume that during the interval $(s_1,s_2]$ both $\tilde X^{\id}$ and $\tilde X^{\tau_1\circ \tau_2}$ satisfy the relaxed conditional uniformity assumption.

Now we show that on the event $A(\delta)$ we can couple the walks such that $\bar X_{\Delta}=\bar Y_{\Delta}$ with high probability. The argument uses a coupling of Berestycki, Schramm, Zeitouni~\cite{berestycki_k}, itself a variant of a beautiful coupling introduced by Schramm~\cite{schramm_transpo}. We first introduce some notation. Let
\[
  \Omega_n:=\{\mathbf (x_1 \geq \dots \geq x_n): x_i \in \{0/n,1/n,\dots,n/n\} \text{ for each }i \leq n \text{ and } \sum_{i\leq n}x_i =1  \}.
\]
Notice that the walks $\bar X$ and $\bar Y$ both take values in $\Omega_n$.

\paragraph{Marginal evolution.} Let us describe the evolution of the random walk $\bar X=(\bar X_s: s = 0,1,\dots)$.
Suppose that $s\geq 0$ and $\bar X_{s}= (x_1,\dots,x_n)$. Now imagine the interval $(0,1]$ tiled using the intervals $(0,x_1],\dots,(0, x_n]$
(the specific tiling rule does not matter).
Initially for $s=0$, and more generally if $s$ is a refreshment time, we select $u \in \{1/n,\dots,n/n\}$ uniformly at random and then call the tile that contains $u$ the \textbf{marked tile}. If $s \geq 1$ is not a refreshment time then the marked tile is the one containing the second marker $y$ of Proposition
\ref{prop:conditional_unif} from the previous step. Either way, we have a distinguished tile (the tile containing the `first marker' at the beginning of each step $s = 0, 1, \ldots$

We now describe the marginal evolution of this tiling for one step. In fact this evolution takes as an input a tiling $\bar X_s$ and a marked tile $I$. The output will be another tiling $\bar X_{s+1}$ and a new marked tile for the next step. Let $I$ be the tile containing the first marker at the beginning of the step, and place $I$ first from left. ($I$ represents the cycle containing the first marker $u$ and we imagine that $u$ is the leftmost point of that tile, i.e., in position $1/n$). Select $v \in \{2/n,\dots,n/n\}$ uniformly at random and let $I'$ be the tile that $v$ falls into. Then there are two possibilities:
\begin{itemize}
 \item if $I' \neq I$ then we merge the tiles $I$ and $I'$. The new tile we created is now marked for the next step.

 \item If $I=I'$ then we split $I$ into two fragments, corresponding to where $v$ falls. Thus, one of size $v-1/n$ and the other of size $|I|-(v-1/n)$. The rightmost one of these two tiles, containing $v$, is now marked for the next step.
Now $\bar X_{s+1}$ is the sizes of the tiles in the new tiling we have created (with additional reordering of tiles in decreasing order). .
\end{itemize}

This defines a transformation $T( \bar X_s, I, v)$. The evolution described above has the law of the projection onto $\Omega_n$ of $\bar X$. Indeed, suppose we apply the transposition $(x,y)$ to $\bar X_s$ in order to obtain $\bar X_{s+1}$. The marked tile at time $s$ corresponds to the cycle of $\bar X_{s}$ containing $x$: if $s$ is a refreshment time then $x\in\{1,\dots,n\}$ is chosen uniformly, otherwise $x$ is the second marker from the previous step.


\paragraph{Coupling.} We now recall the coupling of \cite{berestycki_k}.
Let $s \geq 0$. Suppose that $\bar X_{s}=\bar X=(x_1,\dots,x_n)$ and $\bar Y_s= \bar Y=(y_1,\dots,y_n)$. Then we can differentiate between the entries that are \emph{matched} and those that are \emph{unmatched}: we say that two entries from $\bar X$ and $\bar Y$ are matched if they are of identical size.
Our goal will be to create as many matched parts as possible and as quickly as possible.
When putting down the tilings $\bar X$ and $\bar Y$, associated with $\bar X$ and $\bar Y$ respectively, we will do so in
such a way that all matched parts are to the
right of the interval
$(0,1]$ and the unmatched parts occupy the left part of the interval.


Let $I_{\bar X}$ and $I_{\bar  Y}$ be the respective marked tiles of the tilings $\bar X$ and $\bar Y$ at some step $s \ge 0$,
and let $\hat X, \hat Y$ be the tiling which is the reordering of $\bar X,\bar Y$
in which $I_{\bar X}$ and $I_{\bar Y}$
have been put to the left of the interval $(0,1]$. We assume that at the start of the step, either $I_{\bar X}$ and $I_{\bar Y}$ are both matched to each other, or they are both unmatched. (We will then verify that this property is preserved by the coupling).
Let $a=|I_{\bar X}|$ and let $b=|I_{\bar Y}|$ be the respective
lengths of the marked tiles,
and assume without loss of generality that $a<b$.
Let $v \in \{2/n,\dots,n/n\}$ be chosen uniformly. We will apply $T(\bar X, I_{\hat X}, v)$ to $\hat X$ as we did before, and obtain $\bar X_{s+1}$.
To obtain $\bar Y_{s+1}$ we will also apply the transformation $T$ to it, but with an other uniform random variable $v' \in \{2/n,\dots,n/n\}$ which may differ from $v$. To construct $v'$ we proceed as follows.

If $I_{\hat X}$ is matched (so that $I_{\hat Y}$ is matched to it by assumption) then we take $v'=v$, as in the coupling of Schramm~\cite{schramm_transpo}.
In the case when $I_{\hat X}$ is unmatched (which also implies that $I_{\hat Y}$ is unmatched), we apply to $v$ a measure-preserving
map $\Phi$, defined as follows: for $w \in \{2/n,\dots,n/n\}$ consider the map
\begin{equation}\label{Phi}
  \Phi(w)=
  \begin{cases}
    w  & \text{if $w>b$ or if $2/n\leq w\le\gamma_n +1/n$}, \\
    w - \gamma_n  &  \text{if $a<w\le b$},\\
    w + b-a & \text{if $\gamma_n +1/n<w\le a$},
  \end{cases}
\end{equation}
where $\gamma_n:=\lceil an/2 -1\rceil / n$. (This is contrast with Schramm's original coupling, where $v' = v$ no matter what).
See Figure \ref{fig:bsz_coupling} (top right corner) for an illustration of $\Phi$, from which it should be clear in particular that $\Phi$ is a bijection and hence measure-preserving; this is easy to check. Thus letting $v' = \Phi(v)$ we have that $v'$ has the correct marginal distribution and thus so does $\bar Y_{s+1} = T( \hat Y_s, I_{\bar Y}, v')$. The rest of Figure \ref{fig:bsz_coupling} illustrates the various steps in the coupling as well as the content of Lemma \ref{L:consistent} and Lemma \ref{L:coupling}, as explained in the caption.

\begin{figure}
  \centering
  \includegraphics[scale=.4]{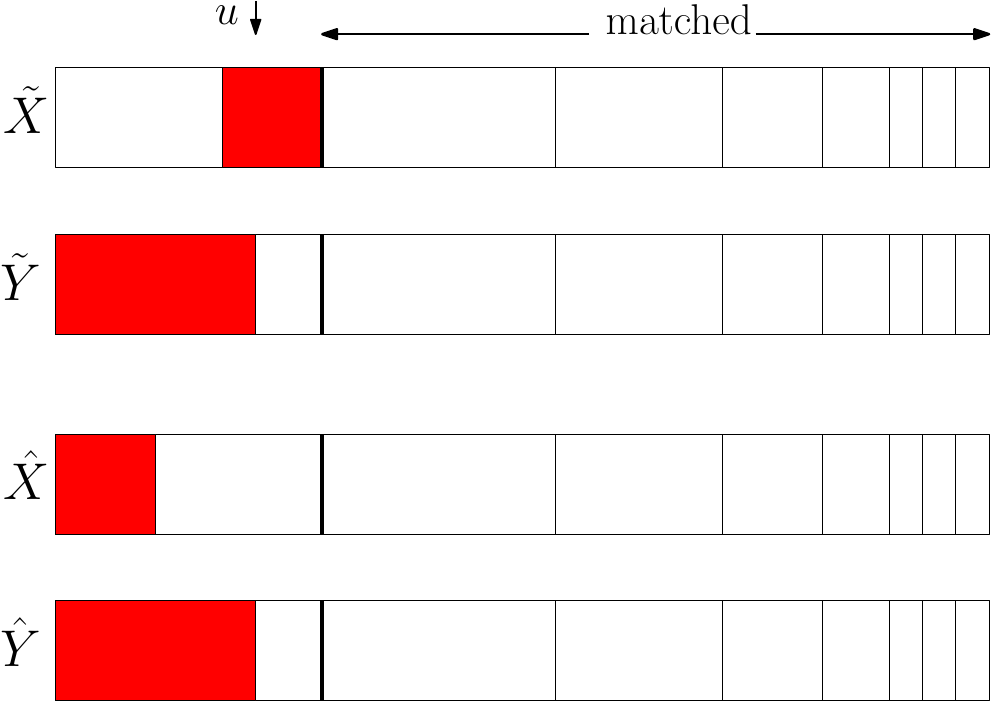} \hspace{1cm}
   \includegraphics[scale=.4]{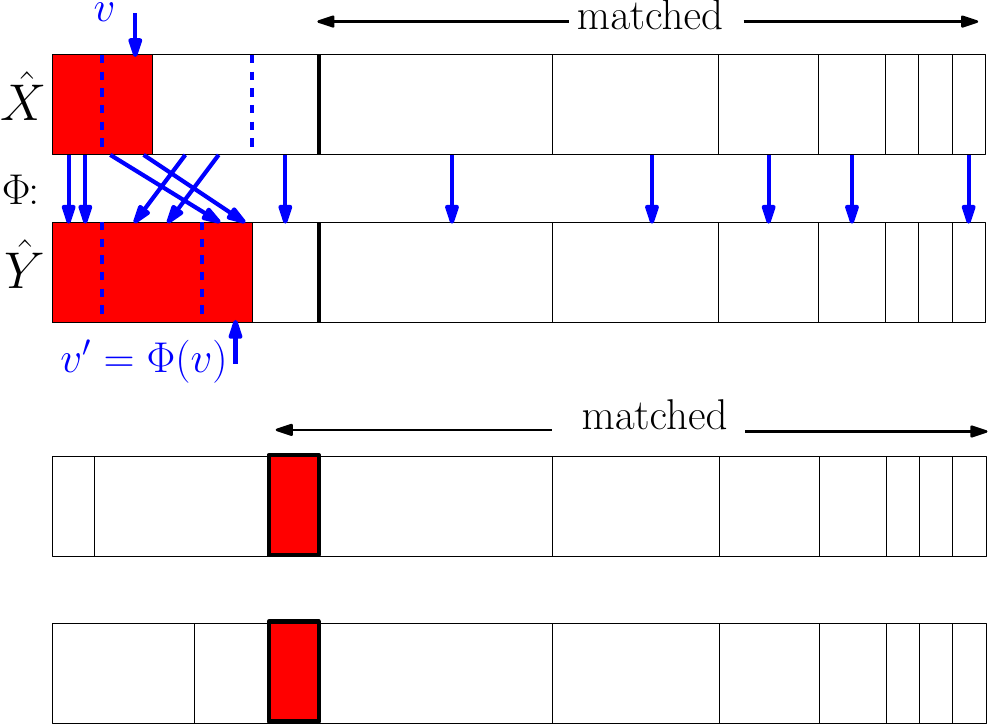}

  \caption{One step of the evolution under the coupling between $\bar X$ and $\bar Y$. The red entries represent the marked tiles. Left: the tilings are rearranged so that both marked tiles (which are here both unmatched) are to the left. Right: the second marker $v$ falls, and defines a marker $v' = \Phi(v)$. In this case we get a fragmentation in both copies. The new marked tile is the rightmost fragment thus created; here they are both matched to each other, verifying in this case the consistency property of the coupling (Lemma \ref{L:consistent}). The total number of unmatched entries has been left unchanged, though. And while the smallest unmatched entry has decreased in size, this has not been by more than a factor of two. This is the content of the key Lemma \ref{L:coupling}.}
  \label{fig:bsz_coupling}
\end{figure}

\medskip Before checking that the coupling is well defined (in the sense that our assumption on the marked tiles, which are needed for the definition of the coupling, remains true throughout), we briefly add a few words of motivation for this definition.

\paragraph{Motivation for the coupling.} The coupling defined above is, as already mentioned above, the same as the one used in \cite{berestycki_k}, which is a modification of a coupling due to Schramm \cite{schramm_transpo}. In Schramm's original coupling, the map $\Phi$ was taken to be the identity, which is natural enough. However this leads to the undesirable property that it is possible for very small unmatched pieces to appear; once these small unmatched pieces appear they remain in the system for a very long time which could prevent coupling. The reason for introducing the map $\Phi$ here and in \cite{berestycki_k} (where it was one of the main innovations) is that it prevents the occurrence of small unmatched pieces: as we will see in Lemma \ref{L:coupling}, the crucial property is that the worst thing that can happen is for the smallest unmatched piece to become smaller by a factor of two, and this only happens with small probability. This means all unmatched pieces remain relatively large, and so they disappear quite quickly (leading in turn to a coupling of the two copies). We start by a proof that the coupling is well defined:

\begin{lemma}\label{L:consistent}
At the end of a step the two marked tiles are either matched to each other or both unmatched.
\end{lemma}

\begin{proof}
The proof consists of examining several cases. If the first marker $u$ was in a matched tile, then whether $v$ falls in the matched or unmatched part, the property holds (if $v$ is unmatched then we attach two unmatched tiles to two matched tiles, so they both become unmatched. If it falls in the matched part, either two tiles of the same size are being attached, or the marked tile splits into two tiles of same size, and the rightmost piece which is the new marked tile matches in both copies).

A similar analysis can be done if the first marked tile was unmatched. The only case which requires an observation is if $v$ falls in the same tile as $I_{\bar X}$ (where we assume, as in the figure, that this is the sorter of the two unmatched pieces), then if $v$ falls in the first half of the tile this results in two matched pieces which are unmarked for the next step and two unmatched pieces which are both marked. If however $v$ falls in the second half of $I_{\bar X}$ then this results in two matched pieces which are marked, and two unmatched pieces which are not marked. (Recall that the marked tile at the next step is the one containing the rightmost fragment).
\end{proof}

\medskip This coupling has several remarkable deterministic properties, as already observed in \cite{berestycki_k}. Chief among those is the fact that \textbf{the number of unmatched entries can only decrease}. Unmatched entries disappear when they are coalesced. In particular they disappear quickly when their size is reasonably large. Hence it is particularly desirable to have a coupling in which unmatched components stay large. The second crucial property of this coupling is that it does not create arbitrarily small unmatched entries: even when unmatched entry is fragmented, the size of the smallest unmatched entry cannot decrease by more than a factor of two. (As these properties hold deterministically given the marked tiles, they do not need to be proved again).
A direct consequence of these properties is the following lemma, which is Lemma 19 from \cite{berestycki_k}.

\begin{lemma}\label{L:coupling}
  Let $U$ be the size of the
  smallest unmatched entry in two partitions $\bar x,\bar y \in \Omega_n$,
  let $\bar x',\bar y'$ be the corresponding partitions after one transposition of the coupling, and let $U'$ be the size of the smallest unmatched entry in $\bar x',\bar y'$. Assume that
  $2^j\le  U < 2^{j+1}$ for some $j \ge 0$.
  Then it is always the case that $U'\ge (1/n)\lfloor nU/2 \rfloor$, and moreover,
  $$
  \P( U' \le 2^j) \le 2^{j+2}/n.
  $$
  Finally, the combined number of unmatched parts may only decrease.
\end{lemma}

\begin{rem}
  In particular, it holds that $U' \ge 2^{j-1}/n$.
\end{rem}

We now explain our strategy. On $A(\delta)$ we will expect that the unmatched components will remain of a size roughly of order at least $\delta$ for a while. In fact we will show that they will stay at least as big as $O(\delta^2)$ for a long time. Unmatched entries disappear when they are merged together. If all unmatched entries are of size at least $\delta^2$, we will see that with probability at least $\delta^8$, we have a chance to reduce the number of unmatched entries in every $3$ steps. Then a simple argument shows that after time $\Delta=\lceil \delta^{-9}\rceil$, $\bar X_\Delta$ and $\bar Y_\Delta$ are perfectly matched with a probability tending to one as $\delta \to 0$.

\begin{lemma}\label{L:big} There is $\delta_0$ such that if $\delta < \delta_0$, during $[0, \Delta]$, both $\bar X_s $ and $\bar Y_s$ always have an entry of size greater than $  \delta \theta(c) $ with probability at least $1- 2\delta^{1/2}$ for all $n$ sufficiently large.
\end{lemma}

\begin{proof}
  Let $\delta_0>0$ be such that $(1-\delta)^{9!} \ge \delta^{1/2}$ for all $\delta \le \delta_0$; we may assume without loss of generality that $\delta \le \delta_0$. Let $Z= (Z_1, \ldots)$ be a Poisson-Dirichlet random variable on $\Omega_\infty$ and let $(Z^*_1,\dots)$ denote the size biased ordering of $Z$. Recall that $Z^*_1$ is uniformly distributed over $[0,1]$, $Z^*_2$ is uniformly distributed on $[0, 1-Z^*_1]$, and so on. For the event $\{Z_1 \le \delta\}$ to occur it is necessary that $Z^*_1\le \delta, Z^*_2 \le \delta/(1-\delta), \ldots, Z^*_{10} \le \delta/(1-\delta)^9$. This has probability at most $\delta^{10}/(1- \delta)^{9!}$.
  Thus
  $$\P(Z_1 \le \delta) \le \frac{\delta^{10} }{ (1- \delta)^{9!} }\le \delta^{9 + 1/2}.$$
  Summing over $\Delta = \lceil \delta^{-9}\rceil$ steps we see that the expected number of times during the interval $[0, \Delta]$ such that $\bar X_s$ or $\bar Y_s$ don't have a component of size at least $\theta(c) \delta n$ is less than $\delta^{1/2}$ as $n \to \infty$ and is thus less than $2\delta^{1/2}$ for $n$ sufficiently large, by Theorem \ref{thm:general_schramm} (note that we can apply the result because this calculation involves only a finite number of components). The result follows.
\end{proof}

We now check that all unmatched components really do stay greater than $\delta^2$ during $[0, \Delta]$. Let $T_\delta$ denote the first time $s$ that either $\bar X_s$ or $\bar Y_s$ have no cycles greater than $ \delta \theta (c) n$ (suppose without loss of generality that $\delta$ is small enough that $\delta^2 \le \delta \theta(c)$).

\begin{lemma}\label{L:big2}
  On $A(\delta)$, for all $s \le T_\delta \wedge \Delta$, all unmatched components stay greater than $\delta^2$ with probability at least $1-O(\delta)
  $, where the constant implied in $O(\delta)$ can depend on $c$ but not on $\delta$.
\end{lemma}

\begin{proof}
  Say that a number $x \in [0,1]$ is in scale $j$ if $2^j/n \le x < 2^{j+1}/n$. For $s\ge 0$, let $U(s)$ denote the scale of the smallest unmatched entry of $\bar X_s, \bar Y_s$.
  Let $j_0$ be the scale of $\delta $, and let $j_1$ be the integer immediately above the scale of $\delta^2$.

  Suppose for some time $s\le T_\delta$, we have $U(s) = j$ with $j_1 \le j \le j_0$, and the marked tile at time $s$ corresponds to the smallest unmatched entry. Then after this transposition we have $U(s+1) \ge j-1$ by the properties of the coupling (Lemma \ref{L:coupling}). Moreover, $U(s+1) = j-1$ with probability at most $r_j = 2^{j+2} / n$. Furthermore, since $s \le T_\delta$, we have that this marked tile merges with a tile of size at least $\theta(c) \delta$ with probability at least $\theta (c)  \delta $ after the transposition. We call the first occurrence a \emph{failure} and the second a \emph{mild success}.

  Once a mild success has occurred, there may still be a few other unmatched entries in scale $j$, but no more than five since the total number of unmatched entries is decreasing, and there were at most six initially.  And therefore if six mild successes occur before a failure, we are guaranteed that $U(s+1) \ge j+1$. We call such an event a \emph{good success}, and note that the probability of a good success, given that $U(s)$ changes scale, is at least $ p_j = 1- 6 r_j / (r_j + \theta(c)  \delta)$. We call $q_j = 1-p_j$.

  Let $\{q_i\}_{i\geq 0}$ be the times at which the smallest unmatched
  entry changes scale, with $q_0$ being the first time the smallest unmatched entry
  is of scale $j_0$.
  Let $\{U_i\}$ denote the scale of the smallest unmatched
  entry at time $q_i$. Introduce a birth-death chain on
  the integers, denoted $v_n$, such that $v_0=j_0$ and
  \begin{equation}
    \label{birthdeath}
    \P(v_{n+1}=j-1|v_n=j)=
    \left\{\begin{array}{ll}
      1& \mbox{\rm if } \, j=j_0\\
      0& \mbox{\rm if }\, j=j_1\\
      q_j&\mbox{\rm otherwise},
    \end{array}
    \right.
  \end{equation}
  and
  \begin{equation}
    \label{birthdeath2}
    \P(v_{n+1}=j+1|v_n=j)=
    \begin{cases}
      p_j, & j>j_1\\
      0, & j=j_1.
    \end{cases}
  \end{equation}
  Then it is a consequence of the above observations that $(U_i, i \ge 1)$ is stochastically dominating $(v_i, i \ge 1)$ for $s \le T_\delta$. Set
  $\tau_j=\min\{n>0: v_n=j\}$.
  An analysis of the birth-death chain defined by
  \eqref{birthdeath}, \eqref{birthdeath2} gives that
  \begin{align*}
    \P^{j_0}(\tau_{j_1}<\tau_{j_0})&  = \frac1{\sum_{j=j_1+1}^{j_0} \prod_{m=j}^{j_0-1} \frac{p_m}{q_m}} \leq \prod_{j=j_1+1}^{j_0-1}
    \frac{q_j}{p_j}
  \end{align*}
  (see, e.g., Theorem (3.7) in Chapter 5 of \cite{durrett_probability}). Note also that $q_j \le 6r_j / (\delta \theta(c))$, so $q_j /p_j \lesssim  r_j /(\delta \theta(c))$.  Moreover all the $r_j$ in this product satisfy $ r_j \lesssim  \delta$.
  Thus, by considering the $10$ terms with lowest index in the product above (and note that for $\delta >0$ small enough, there are at least $10$ terms in this product), we deduce that
  $ \P^{j_0}(\tau_{j_1}<\tau_{j_0})$ decays faster than $O(\delta)^{10}$. Since $T_\delta \wedge \Delta \leq \Delta = \lceil \delta^{-9}\rceil $ we conclude that the probability that $U(s) = j_1$ before $T_\delta \wedge \Delta$ is at most $O(\delta) $.
\end{proof}

We are now going to prove that on the event $A(\delta)$, after time $\Delta$ there are no unmatched entries with probability tending to one as $n\rightarrow\infty$ and $\delta \to 0$. The basic idea is again to exploit that there are initially at most six unmatched parts, and this number cannot increase. We need a few preparatory lemmas which construct a scenario which lead to a coagulation of two unmatched entries in each copy, in three steps. Let $T'_\delta $ be the first time one of the unmatched entries is smaller than $\delta^2$. Let $\cF_s$ denote the filtration generated by $(\bar X_1, \ldots, \bar X_s)$ including the marked tiles at the end of each step up to time $s$. Let $\cK_s$ be the event that step $s$ results in two unmatched entries being merged in both copies, so our first goal (achieved in Lemma \ref{L:step3}) will be to get a lower bound on the probability of $\cK_s$.

\medskip \noindent \textbf{Step 1.} We show that with good probability both marked tiles are unmatched at the end of a step (and thus also at the beginning of the next step).

\begin{lemma}\label{L:Marked_unmatched}
Let $\cM_s$ be the event that at the end of step $s$, both marked tiles are unmatched.
 Then
$$
\P(\cM_s \cup \cK_{s} | \cF_{s-1}; T'_\delta \ge s ) \ge \delta^2.
$$
\end{lemma}

\begin{proof}
%
Let $u,v$ be the two markers for step $s$. If the tile containing $u$ was matched, then it suffices for $v$ to fall in an unmatched tile (then $\cM_s$ occurs), which occurs again with probability at least $\delta^2$. If however the tile containing $u$ was unmatched, the copy which contains the smallest of these two unmatched tiles necessarily contains at least another unmatched tile. It then suffices for $v$ to fall in that tile. Indeed if we are very lucky and the other copy also just happen to have two unmatched entries this might lead to a reduction in the number of unmatched entries, in which case $\cK_{s}$ has occurred. Otherwise we have simply shuffled the unmatched entries and $v$ is now in an unmatched entry, so $\cM_s$ holds. Either way, the conditional probability is at least $\delta^2$.
\end{proof}

\medskip \noindent \textbf{Step 2.} We show that if the marked tiles are unmatched, with good probability we can get to a ``balanced configuration" where both copies contain at least two unmatched tiles, and that the marked tiles at the end of the step are both unmatched.

\begin{lemma}\label{L:step2}
Suppose $s$ is not a refreshment time.  Let $\cB_s$ denote the event that $\bar X_s$ and $\bar Y_s$ contain at least two unmatched entries each, and that the second marker is in one of these unmatched tiles for both $\bar X_s, \bar Y_s$ at the end of the step (i.e., $\cM_{s}$ holds). Then
  \begin{equation}\label{eq:step2notR}
\P(\cB_s \cup \cK_{s} | \cF_{s-1}; T'_\delta \ge  s; \cM_{s-1} ) \ge \delta^2.
  \end{equation}
 Suppose now that $s$ is a refreshment time. Then
  \begin{equation}\label{eq:step2R}
\P(\cB_s  | \cF_{s-1}; T'_\delta \ge s) \ge \delta^4.
  \end{equation}
\end{lemma}

\begin{proof}
Let $u,v$ be the two markers for step $s$. Suppose first that $s$ is not a refreshment time, so we aim to prove \eqref{eq:step2notR}. We treat several cases, according to whether
$\cB_{s-1}$ holds or not. We start by assuming that $\cB_{s-1}$ does not hold. The idea is that in that case, at time $s-1$, one copy (say $\bar Y_{s-1}$) has one unmatched entry, while the other one has at least three. It then suffices to fragment the unmatched entry in $\bar Y_{s-1}$ and to coagulate the other two entries in $\bar X_{s-1}$. Since $\cM_{s-1}$ holds, and $s$ is not a refreshment time, it suffices for $v$ (the marker corresponding to the copy which has the smallest unmatched entry, which is necessarily $\bar X_{s-1}$) to fall in any of the other unmatched entries of $\bar X_{s-1}$: this necessarily results in a balanced configuration. Note also that this always results in both marked tiles to be unmatched at the end of the step, so $\cB_s$ indeed holds in that case. Moreover, this event has probability at least $\delta^2$ since $T'_\delta \ge s$.

Suppose now that $\cB_{s-1}$ holds. Then let us show directly $\cK_s$ can occur with good probability. Indeed, if the second marker $v'  = \Phi(v)$ (this is the marker associated with the copy, say $\bar Y_s$, that contains the larger of the two marked unmatched tiles) falls in another unmatched tile of $\bar Y_s$, then in this
 case a coagulation of two unmatched entries is guaranteed to occur in both copies. Hence $\cK_{s}$ occurs with probability at least $\delta^2$. Either way,  \eqref{eq:step2notR} is proved.

Now suppose that $s$ is a refreshment time. In that case it suffices to require that the first marker falls in an unmatched tile (which occurs with probability $\delta^2$) and from then on we argue exactly as in the proof of \eqref{eq:step2notR} to obtain a proof of \eqref{eq:step2R}. All in all the lemma is proved.
\end{proof}

We point out that, combining Lemmas \ref{L:Marked_unmatched} and \ref{L:step2}, regardless of whether $s$ is a refreshment time, $\P( \cB_s | \cF_{s-1}) \ge \delta^4$.

\medskip \noindent \textbf{Step 3.} Having reached a balanced configuration with one marked unmatched entry in both copies, we show that a coagulation of two unmatched entries in both copies has a good chance of occurring. In that case, the number of unmatched entries has decreased by two or four.

\begin{lemma}\label{L:step3} We have
$$
\P(\cK_s | \cF_{s-1}; \cB_{s-1}; T'_\delta \ge s) \ge \delta^4.
$$
\end{lemma}

\begin{proof}
We again need to distinguish between the cases where $s$ is refreshment time or not. If not, then since $\cB_{s-1}$ holds, then the first marker is in an unmatched tile for both copies. If the second marker $v'$ corresponding to the copy with the larger of these two unmatched tiles falls in a different unmatched tile (which has probability at least $\delta^2$) then a coagulation is guaranteed to occur in both copies so $\cK_{s}$ holds.

If $s$ is refreshment time, then the same argument applies, but the first marker must first fall in an unmatched component (which has probability at least $\delta^2$ since $T'_\delta \ge s$). This gives a lower bound of $\delta^4$ on the probability of $\cK_s$, as desired.
\end{proof}

Combining these three steps, it is now relatively easy to deduce the following:

\begin{lemma}\label{lemma:X_not_Y}
  We have that for all $\delta>0$ small enough
  \[
    \lim_{\delta \rightarrow 0}\limsup_{n \rightarrow \infty}\P(\bar X_\Delta \neq \bar Y_\Delta|A(\delta)) =0
  \]
\end{lemma}
\begin{proof}
%
  Initially there are at most $6$ unmatched entries. Due to parity there can be either $6,4$ or $0$ unmatched entries (note in particular that 2 is excluded, as a quick examination shows that no configuration can give rise to two unmatched entries). Furthermore, form the properties of the coupling, the number of unmatched entries either remains the same or decreases at each step. Once all the entries are matched they remain matched thereafter.

We have just shown that in any sequence of three transpositions, the probability that the number of unmatched decreases is at least $\delta^8$, unless $T'_\delta$ occurs during this sequence. Let $Z$ be a binomial random variable with parameters $m = \lfloor (\Delta -1)/3 \rfloor $ and $p = \delta^8$. Thus the event that $\{X_\Delta \neq Y_\Delta\}$ implies that there has been at most one success (i.e. $Z \le 1$, so
\begin{align*}
\P(X_\Delta \neq Y_\Delta | A(\delta)) &\le \P( Z\le 1) + \P( T'_\delta\le \Delta | A(\delta))\\
& \le (1- p)^{m} + m p(1-p)^{m-1} + O(\delta).
\end{align*}
Since $\Delta \gtrsim \delta^{-9}$ and $p = \delta^8$, the first two terms tend to 0 as $\delta \to 0$, which proves Lemma \ref{lemma:X_not_Y}.
\end{proof}

\subsubsection{Coupling for \texorpdfstring{$(s_2,s_3]$}{(s2,s3]}}

The walks $\tilde X^{\id}$ and $\tilde X^{\tau_1\circ \tau_2}$ are uniformly distributed on their conjugacy class. Thus one can couple $\tilde X^{\id}$ and $\tilde X^{\tau_1\circ \tau_2}$ so that
\begin{itemize}
  \item on the event $A(\delta)^c$ we have that $d(\tilde X^{\id}_{s_2}, \tilde X^{\tau_1\circ\tau_2}_{s_2})=2$,
  \item we have that using Lemma \ref{lemma:X_not_Y}
    \[
      \liminf_{\delta \downarrow 0}\liminf_{n \rightarrow \infty}\P(\tilde X^{\id}_{s_2}= \tilde X^{\tau_1\circ\tau_2}_{s_2}|A(\delta))=1,
    \]
  \item on the event $\{\tilde X^{\id}_{s_2}\neq \tilde X^{\tau_1\circ\tau_2}_{s_2}\}$, note that the walks $\bar X$ and $\bar Y$ have at most $6$ unmatched entries. Hence the coupling is such that $d(\tilde X^{\id}_{s_2}, \tilde X^{\tau_1\circ\tau_2}_{s_2})\leq 4$ no matter what.
\end{itemize}

Combining this with Lemma \ref{L:unmatched_large_0} we have just shown the following lemma.

\begin{lemma}\label{lemma:Sn_coupling}
  There exists a coupling of $\tilde X^{\id}$ and $\tilde X^{\tau_1 \circ \tau_2}$ such that
  \[
    \limsup_{\delta \downarrow 0}\limsup_{n \rightarrow \infty}\E[d(\tilde X^{\id}_{s_2}, \tilde X^{\tau_1\circ\tau_2}_{s_2})] \leq 2(1-\theta(c)^4)
  \]
\end{lemma}

The theorem now follows immediately:

\begin{proof}[Proof of Theorem \ref{thm:curv}]
It remains to see the coupling during the time interval $(s_2,s_3]$. During this time interval we apply the same transpositions to both $\tilde X^{\id}$ and $\tilde X^{\tau_1\circ\tau_2}$ which keeps their distance constant throughout $(s_2,s_3]$. Thus we have that
\[
  d(X^{\id}_t,X^{\tau_1\circ\tau_2}_t)=d(\tilde X^{\id}_{s_3},\tilde X^{\tau_1\circ \tau_2}_{s_3})=d(\tilde X^{\id}_{s_2},\tilde X^{\tau_1\circ \tau_2}_{s_2}).
\]
Thus using Lemma \ref{lemma:Sn_coupling} we see that \eqref{goal0} holds which finishes the proof.
\end{proof}

\appendix
\section{Lower bound on mixing}\label{S:lb}

In this section we give a proof of the lower bound on $\tmix(\delta)$ for some arbitrary $\delta\in (0,1)$. This is for the most part a well-known argument, which shows that the number of fixed points at time $(1-\epsilon) \tmix$ is large. In the case of random transpositions or more generally of a conjugacy class $\Gamma$ such that $|\Gamma|$ is finite, this follows easily from the coupon collector problem. When $|\Gamma|$ is allowed to grow with $n$, we present here a self-contained argument for completeness.

Let $\Gamma\subset \Sn$ be a conjugacy class and set $k= k(n) = |\Gamma|$.

\begin{lemma}
  We have that for any $\epsilon \in (0,1)$,
  \[
    \lim_{n \rightarrow \infty}d_{TV}((1-\epsilon)\tmix)=1
  \]
\end{lemma}
\begin{proof}
  Let $K_m \subset \Sn$ be the set of permutations which have at least $m$ fixed points. Recall that $\mu$ is the invariant measure, which is a uniform probability measure on $\Sn$ or $\An$ depending on the parity of $\Gamma$. Let $U$ denote the uniform measure on $\Sn$. Either way,
  $$
  \mu(K_m) \le 2 U(K_m).
  $$
  Now, $U(K_m) \to \sum_{j=m}^\infty e^{-1} \frac1{j!} $ as $n \to \infty$, hence we deduce that
  \begin{equation}\label{limsupfixedpoints}
    \limsup_{m\to \infty} \limsup_{n\to \infty} \mu(K_m) = 0.
  \end{equation}

  Fix $\beta >0$ and let $$
  t_\beta =  \frac1{k} n (\log n - \log \beta).$$
  Assume that $\beta$ is such that $t_\beta$ is an integer. For each $i \geq 0$, $\gamma_i$ write $N(\gamma_i) \subset \{1,\dots,n\}$ for the set of non-fixed points of $\gamma_i$. Then we have that for each $i \geq 0$, $|N(\gamma_i)|= k$ and further $\{N(\gamma_i)\}_{i=1}^\infty$ are i.i.d. subsets of $\{1,\dots,n\}$ chosen uniformly among the subsets of size $k=|\Gamma|$.

  Consider for $1\leq i \leq n$ the event $A_i$ that the $i$-th card is not collected by time $t_\beta$, that is $i \notin \bigcup_{\ell=1}^{t_\beta} N(\gamma_\ell)$. Thus for $1\leq i_1<\dots< i_\ell\leq n$ and $\ell \leq n-k$,
  \[
    \P(A_{i_1} \cap \dots \cap A_{i_\ell})= \left(\frac{\binom{n-\ell}{k}}{\binom{n}{k}}\right)^{t_\beta}.
  \]
  Let $N=N(n)\in \N$ be increasing to infinity such that $N^2=o(n)$ and $N=o(n^2 k^{-2})$. Then by the inclusion-exclusion formula we have that
  \begin{equation}\label{eq:inclusion_exclusion}
    \P(A_1\cup \dots \cup A_N) = \sum^{N}_{\ell=1} (-1)^{\ell+1} \binom{n}{\ell}\left(\frac{\binom{n-\ell}{k}}{\binom{n}{k}}\right)^{t_\beta}.
  \end{equation}
  Writing out the fraction of binomials on the right hand side we have
  \[
    \left(1-\frac{k}{n-\ell}\right)^{\ell t_\beta} \leq \left(\frac{\binom{n-\ell}{k}}{\binom{n}{k}}\right)^{t_\beta} \leq \left(1-\frac{k}{n}\right)^{\ell t_\beta}.
  \]
  Now $-x/(1-x) \leq \log(1-x) \leq -x$ for $x \in (0,1)$ thus we have that
  \begin{equation}\label{eq:binom_frac}
    \exp\left(- \frac{\ell k t_\beta}{n-k-\ell}\right) \leq \left(\frac{\binom{n-\ell}{k}}{\binom{n}{k}}\right)^{t_\beta} \leq \exp\left(- \frac{\ell k t_\beta}{n}\right).
  \end{equation}
  On the other hand we have that
  \begin{equation}\label{eq:binomial_est}
    \frac{(n-\ell)^\ell}{\ell!} \leq \binom{n}{\ell} \leq \frac{n^\ell}{\ell!}.
  \end{equation}
  Note that $ne^{-t_\beta k/n}=\beta$, then combining \eqref{eq:binom_frac} and \eqref{eq:binomial_est} we get
  \begin{equation}\label{eq:Sv_bounds}
    \left(1-\frac{\ell}{n}\right)^\ell \exp\left(-\frac{k(k+\ell)\ell t_\beta}{n(n-k-\ell)}\right) \frac{\beta^\ell}{\ell!}\leq \binom{n}{\ell}\left(\frac{\binom{n-\ell}{k}}{\binom{n}{k}}\right)^{t_\beta} \leq \frac{\beta^\ell}{\ell!}.
  \end{equation}
  Let us lower bound the error term on the left hand side of \eqref{eq:Sv_bounds}. First $(1-\ell/n)^\ell \geq e^{-\ell^2/(n-\ell)}$, hence it follows that
  \[
    \inf_{\ell \leq N}\left(1-\frac{\ell}{n}\right)^\ell \exp\left(-\frac{k(k+\ell)\ell t_\beta}{n(n-k-\ell)}\right) \geq \inf_{\ell \leq N}\exp\left(-\frac{\ell^2}{n-\ell}-\frac{k(k+\ell)\ell t_\beta}{n(n-k-\ell)}\right).
  \]
  It is easy to see that the right hand side above converges to $1$ as $n \rightarrow \infty$. Using this and \eqref{eq:Sv_bounds} it follows that
  \[
    \lim_{n \rightarrow \infty} \sum^{N}_{\ell=1} (-1)^{\ell+1} \binom{n}{\ell}\left(\frac{\binom{n-\ell}{k}}{\binom{n}{k}}\right)^{t_\beta} = \lim_{n \rightarrow \infty} \sum_{\ell=1}^N (-1)^{\ell+1}\frac{\beta^{\ell}}{\ell!}= 1- e^{-\beta}.
  \]
  For integers $a<b$ let Let $K_{[a,b]} = A_{a+1} \cup A_{a+2} \cup \ldots A_{b}$. Then we have shown
  $$
  \liminf_{n \to \infty} \P( X_{t_\beta} \in K_{[1,N]} ) \ge 1- e^{-\beta}.
  $$
  Likewise, for any $j < \lfloor n/N\rfloor$,
  $$
  \liminf_{n \to \infty} \P( X_{t_\beta} \in K_{[jN,(j+1)N]} ) \ge 1- e^{-\beta}.
  $$
  Hence
  $$
  \liminf_{n \to \infty} \P( X_{t_\beta}  \in \cap_{j=1}^m  K_{[jN,(j+1)N]} ) \ge 1- m e^{-\beta}.
  $$
  Let $\epsilon>0$. Then for any $\beta>0$, if $t = (1-\epsilon)\tmix $ then $t< t_\beta$ for $n$ sufficiently large, and hence
  $$
  \liminf_{n \to \infty} \P( X_{t}  \in \cap_{j=1}^m  K_{[jN,(j+1)N]} ) = 1.
  $$
  But it is obvious that $ \cap_{j=1}^m  K_{[jN,(j+1)N]} \subset K_m$ and hence for $t = (1-\epsilon) \tmix$,
  \begin{equation}
    \liminf_{n \to \infty} \P( X_t  \in K_m) =1.
  \end{equation}
  Comparing with \eqref{limsupfixedpoints} the result follows.
\end{proof}

		\section{Proof of Theorem \ref{thm:general_schramm}}

		Let $\Gamma \subset \Sn$ be a conjugacy class with cycle structure $(k_2,k_3,\dots)$. Let $X=(X_t:t = 0,1,\dots)$ be a random walk on $\Sn$ which at each step applies an independent uniformly random element of $\Gamma$. Let $\rho=\sum_j (j-1)k_j$ and let $\tilde X$ be the transposition walk associated to the walk $X$ using \eqref{eq:tilde_X_defn}. In particular for $t \geq 0$, $\tilde X_{t\rho}=X_t$. Finally let $Z=(Z_1,Z_2,\dots)$ denote a Poisson--Dirichlet random variable.

		For convenience we restate Theorem \ref{thm:general_schramm} here.

		\begin{thm}\label{thm:schramm_repeat}
			Let $s \geq 0$ be such that $sk/(n\rho) \rightarrow c$ for some $c>c_\Gamma$. Then for each $m \in \N$ we have that as $n \rightarrow \infty$,
			\[
				\left( \frac{\bar \X_1(\tilde X_s)}{\theta(c)},\dots,\frac{\bar \X_m(\tilde X_s)}{\theta(c)}\right) \rightarrow (Z_1,\dots,Z_m)
			\]
			in distribution where $\theta(c)$ is given by \eqref{eq:theta2}.
		\end{thm}

		The proof of this result is very similar to the proof of Theorem 1.1 in \cite{schramm_transpo}. We give the details here.

		Recall the hypergraph process $H=(H_t:t = 0,1,\dots)$ associated with the walk $X$ defined in Section \ref{SS:hyper}. Analogously let $\tilde G=(\tilde G_t:t = 0,1,\dots)$ be a process of graphs on $\{1,\dots,n\}$ such that the edge $\{x,y\}$ is present in $\tilde G_t$ if and only if the transposition $(x,y)$ has been applied to $\tilde X$ prior to and including time $t$. Hence we have that for each $t=0,1,\dots$, $\tilde G_{t\rho}=H_t$.

		Recall that $\tilde X$ satisfies conditional uniformity as described in Proposition \ref{prop:conditional_unif}.
		Using the graph process $\tilde G$ above and the conditional uniformity of $\tilde X$ the following lemma, which is the analogue of Lemma 2.4 in \cite{schramm_transpo}, follows almost verbatim from Schramm's arguments.

		\begin{lemma}\label{lemma:schramm_key}
			Let $s \geq 0$ be such that $sk/(n\rho) \rightarrow c$ for some $c>c_\Gamma$ and let $\epsilon\in (0,1/8)$. Let $M=M(\epsilon,n,s)$ be the minimum number of cycles of $\tilde X_s$ which are needed to cover at least $(1-\epsilon)$ proportion of the vertices in the giant component of $\tilde G_s$. Then for $\alpha \in (0,1/8)$ we have that
			\[
				\limsup_{n \rightarrow \infty} \P(M>\alpha^{-1}|\log(\alpha\epsilon)|^2) \leq C \alpha
			\]
			for some constant $C$ which does not depend on $\alpha$ nor $\epsilon$.
		\end{lemma}

		Henceforth fix some time $s \geq 0$ such that $sk/(n\rho) \rightarrow c$ for some $c>c_\Gamma$. Fix $\epsilon \in (0,1/8)$ and define
		\begin{align*}
			\Delta &:=\lfloor\epsilon^{-1}\rfloor\\
			s_0 &:= s-\Delta.
		\end{align*}
		For $t =0,\dots, \Delta$ define $\bar X_t=\X(\tilde X_{s_0+t})$. We can assume that for $t \leq \Delta$, $ \tilde X_{s_0+t}$ satisfies the relaxed conditional uniformity assumption described in Definition \ref{defn:relaxed}. Indeed by making this assumption we are disregarding the constraint on the transpositions described in Proposition \ref{prop:conditional_unif} applied to $\tilde X_t$ for $t=s_0,\dots,s$. However the probability that we violate this constraint is at most $2\Delta k/n$.

		Colour an element of $\bar X_0=\X(\tilde X_{s_0})$ green if the cycle whose renormalised cycle length of this element lies in the giant component of $\tilde G_{s_0}$. We colour all the other elements of $\bar X_0$ red. Thus asymptotically in $n$, the sum of the green elements is $\theta(c)$ and the sum of the red elements is $1-\theta(c)$. In the evolution of $(\bar X_t:t= 0,1,\dots)$ we keep the colour scheme as follows. If an element fragments, both fragments retain the same colour. If we coagulate two elements of the same colour then the new element retains the colour of the previous two elements. If we coagulate a green element and a red element, then the colour of the resulting element is green.

		Define $\bar X'=(\bar X'_t :t =0,\dots, \Delta)$ and $\bar X''=(\bar X''_t:t =0,\dots, \Delta)$ as follows. Initially $\bar X'_0=\bar X''_0=\bar X_0$. Apply the same colouring scheme to $\bar X'$ and $\bar X''$ as we did to $\bar X$. Each step evolution is described as follows. Then the walks evolve as follows.
		\begin{itemize}
			\item $\bar X'_t$: Evolves the same as $\bar X$ except we ignore any transition which involves a red entry.
			\item $\bar X''_t$: Evolves the same as $\bar X'$ except that the markers $u,v$ used in the transitions of $\bar X''$ are distributed uniformly on $[0,1]$.
		\end{itemize}

		Lemma \ref{lemma:hyper_k_cycle} states that the second largest component of $\tilde G_{s_0}$ has size $o(n)$. Hence, initially each red element has size $o(1)$ as $n\rightarrow\infty$. Now $\Delta$ does not increase with $n$, hence for any $s =0,1,\dots,\Delta$, we are unlikely to make a coagulation (or fragmentation) in $\bar X'_s$ without coagulating (or fragmenting) entries of $\bar X_s$ of similar size. Similar considerations for the processes $\bar X'$ and $\bar X''$ leads to the following lemma.

		\begin{lemma}\label{lemma:couple_prime}
			There exists a coupling between the walks $\bar X$ and $\bar X'$, and between $\bar X'$ and $\bar X''$ such that for each $\eta>0$,
			\[
				\lim_{n \rightarrow \infty}\P\left(\sup_{i \in \N}| \bar X_i(\Delta)- \bar X'_i(\Delta)|>\eta \right)=\lim_{n \rightarrow \infty}\P\left(\sup_{i \in N}| \bar X'_i(\Delta)- \bar X''_i(\Delta)|>\eta \right)=0.
			\]
		\end{lemma}

		Using the preceding lemma, it suffices now to find an appropriate coupling between $\bar X''$ and $Z$. To do this we modify Schramm's coupling in \cite{schramm_transpo}. First we let $\{J_1,\dots,J_L\}$ be the set of times $s \in \{0,\dots,\Delta\}$ such that that $\bar X''_{s-1} \neq \bar X''_{s}$. It is easy to see that $\lim_{n \rightarrow \infty}\P(L>\sqrt \Delta) =1$ and henceforth we will condition on the event that $\{L>\sqrt \Delta\}$ and set $\Delta'=\lfloor \sqrt\Delta\rfloor$. Define a process $\bar Y=(\bar Y_t:t=0,\dots,\Delta')$ as follows. Initially $\bar Y_0=\bar X''_0$. For $t =1,\dots,\Delta'$ we let $\bar Y_t$ be $X''_{J_t}$ renormalised so that $\sum_i \bar Y_i(t) =1$ where $\bar Y_i(t)$ is the $i$-th element of $\bar Y_t$.

		We define a process $\bar Z=(Z_t:t=0,1,\dots,\sqrt \Delta)$ as follows. Initially $Z_0$ has the distribution of a Poisson--Dirichlet random variable, independent of $\bar Y$. Then for $t=1,\dots,\Delta'$ define $\bar Z_t$ by applying the coupling in Section \ref{subsec:I_2} to $\bar Y$ and $\bar Z$ but with the following modifications:
		\begin{itemize}
			\item the markers $u,v \in [0,1]$ are taken uniformly at random,
			\item we always take $v'=v$,
			\item we modify the definition of a refreshment time: $s$ is a refreshment time if either $J_{s-1} + 2 \leq J_s$ or $J_s+s_0$ is a refreshment time in the sense of Definition \ref{def:refreshment},
			\item when a marked tile of size $a$ fragments, it creates a tile of length $v$ and a tile of length $a-v$. We mark the tile of length $a-v$.
		\end{itemize}
		It is not hard to check that the Poisson--Dirichlet distribution is invariant under this evolution and hence we have that for each $t=0,1,\dots,\Delta'$, $Z_t$ has the law of a Poisson--Dirichlet.

		Our coupling agrees with the coupling in \cite[Section 3]{schramm_transpo} when $\Gamma=T$ is the set of all transpositions. In this case each time $s$ is a refreshment time and hence the marked tile at time $s$ is always chosen by the marker $u$. One can adapt the arguments in Chapter 3 of Schramm's paper to our case by using the following idea. Note first that all the estimates of Schramm apply at $s$ when $s$ is a refreshment time. When $s$ is not a refreshment time and Schramm considers the event that the marker $u$ at time $s$ falls inside an unmatched tile, instead we consider the event that the marker $v$ at time $s-1$ falls inside an unmatched tile. By the properties of the coupling, this guarantees that at time $s$ the marked tile is unmatched.

		Adapting Schramm's arguments leads to the following lemma, which is the analogue of \cite[Corollary 3.4]{schramm_transpo}.

		\begin{lemma}
			Define
			\[
				N^0:= \#\{i \in \N: \bar Y_i(0)>\epsilon\} +  \#\{i \in \N: \bar Z_i(0)>\epsilon\}
			\]
			and let
			\[
				\bar \epsilon:= \epsilon + \sum_{i=1}^{\infty} \bar Y_i(0) \1_{\{ \bar Y_i(0)<\epsilon\}}  + \sum_{i=1}^\infty \bar Z_i(0)\1_{\{\bar Z_i(0)<\epsilon\}}.
			\]
			Define the event
			\[
				\mathcal B =  \left\{ \bar\epsilon^{4/5} \leq \frac{1}{\Delta'}\leq \frac{\bar\epsilon^{1/5}}{N^0 \vee 1}  \right\}.
			\]
			Let $q \in \{1,\dots,\Delta'\}$ be distributed uniformly, independent of the processes $\bar Y$ and $\bar Z$. Then we have that for each $\rho>0$,
			\[
				\P(\sup_{i \in \N} |\bar Y_i(q)-\bar Z_i(q)| > \rho ) \leq C \frac{\P(\mathcal B)}{\rho \log \Delta'}
			\]
			for some constant $C>0$.
			\label{lemma:schramm_34}
		\end{lemma}

		Using Lemma \ref{lemma:schramm_34} it suffices to show that $\P(\mathcal B)/\log \Delta' \rightarrow 0$ as $\epsilon \downarrow 0$. The following lemma shows a stronger result.

		\begin{lemma}
			\label{lemma:schramm_condition_check}
			Suppose that $\mathcal B$ is defined as in Lemma \ref{lemma:schramm_34}, then
			\[
				\lim_{\epsilon \downarrow 0} \P(\mathcal B) = 1.
			\]
		\end{lemma}
		\begin{proof}
			Let
			\begin{align*}
				\mathcal B_1 &:=  \left\{ \bar\epsilon^{4/5} \leq \frac{1}{2}\epsilon^{1/2} \right\}\\
				\mathcal B_2 &:=  \left\{  2\epsilon^{1/2}\leq \frac{\bar\epsilon^{1/5}}{N^0 \vee 1}  \right\}
			\end{align*}
			Now as $(1/2)\epsilon^{-1/2} \leq \Delta' \leq 2\epsilon^{-1/2}$ we have that $\mathcal B\supset\mathcal B_1 \cap \mathcal B_2$. First let us bound $\P(\mathcal B_1^c)$. Note that on the event $\mathcal B_1^c$ we have that $\bar \epsilon > 2^{-5/4}\epsilon^{5/8}$.
			Note that a size biased sample from a Poisson--Dirichlet random variable has a uniform law on $[0,1]$. Hence it follows that
			\[
				\E\left[\sum_{i=1}^\infty \bar Z_i(0) \1_{\{\bar Z_i(0)<\epsilon\}}\right] = \epsilon
			\]
			and thus
			\[
				\P\left(  \sum_{i=1}^\infty \bar Z_{i}(0) \1_{\{\bar Z_i(0)<\epsilon\}} > \epsilon^{5/6}\right) \leq \frac{\E\left[\sum_{i=1}^\infty \bar Z_i(0) \1_{\{\bar Z_i(0)<\epsilon\}}\right] }{\epsilon^{5/6}} \leq \epsilon^{1/6}.
			\]
			Next consider the random variable $M$ in Lemma \ref{lemma:schramm_key} at time $s_0=s-\Delta$ where we recall that $\bar Y(0)=\X(\tilde X_{s_0})$. We have that
			\[
				\sum_{i=1}^\infty Y_i(0) \1_{\{\bar Y_i(0): \bar Y_i(0)<\epsilon\}} \leq \epsilon (M+1)
			\]
			Then applying Lemma \ref{lemma:schramm_key} at time $s_0$ we have that
			\[
				\P\left( \sum_{i=1}^\infty \bar Y_i(0) \1_{\{\bar Y_i(0)<\epsilon\}} > \epsilon^{5/6} \right) \leq \P(M> \epsilon^{-1/5}) \leq C \epsilon^{1/6}
			\]
			for some constant $C>0$. Hence it follows that for $\epsilon>0$ small
			\[
				\P(\mathcal B^c_1)=\P(\bar \epsilon > 2^{-5/4} \epsilon^{5/8}) \leq \P(\bar \epsilon > \epsilon^{5/6}) \leq \epsilon^{1/6} + C\epsilon^{1/6}.
			\]
			which shows that $\P(\mathcal B_1) \rightarrow 1$ as $\epsilon \downarrow 0$.

			Now we bound $\P(\mathcal B^c_2)$. Firstly we use the bound $\bar \epsilon \geq \epsilon$ and so we are left to bound $N^0$ from above. Using the stick breaking construction of Poisson--Dirichlet random variables (see for example \cite[Definition 1.4]{berestyckibook}) one can show that
			\[
				\P\left( \#\left\{ i\in \N: Z_i(0) > \epsilon \right\}>\epsilon^{-1/4} \right) \leq C' \epsilon
			\]
			for some constant $C'>0$. On the other hand we have that
			\[
				\#\left\{i \in \N: Y_i(0)>\epsilon  \right\} \leq M
			\]
			and hence using Lemma \ref{lemma:schramm_key} we obtain
			\[
				\P\left( \#\left\{ i\in \N: Y_i(0) > \epsilon \right\}>\epsilon^{-1/4} \right) \leq C'' \epsilon^{1/5}
			\]
			for some constant $C''>0$. Hence it follows that $\P(\mathcal B^c_2) \leq C'\epsilon + C'' \epsilon^{1/5}$ and the result now follows.
		\end{proof}
Theorem \ref{thm:schramm_repeat} now follows from Lemma \ref{lemma:couple_prime}, Lemma \ref{lemma:schramm_34} and Lemma \ref{lemma:schramm_condition_check}.

	\end{document}